\documentclass[12pt]{amsart}
\usepackage{amssymb, latexsym, euscript}
\def\sD{{\mathfrak D}}      
   \def\sH{{\mathfrak H}}   
      \def\sL{{\mathfrak L}}
\def\sM{{\mathfrak M}}   \def\sN{{\mathfrak N}}   
      
\def\sS{{\mathfrak S}}

      \def\dC{{\mathbb C}}
\def\dD{{\mathbb D}}

   \def\dN{{\mathbb N}}   
      \def\dR{{\mathbb R}}

\def\cA{{\mathcal A}}   \def\cB{{\mathcal B}}   \def\cC{{\mathcal C}}
\def\cD{{\mathcal D}}      
   \def\cH{{\mathcal H}}   
   \def\cK{{\mathcal K}}   
\def\cM{{\mathcal M}}   \def\cN{{\mathcal N}}   
      
\def\cS{{\mathcal S}}   \def\cT{{\mathcal T}}   \def\cU{{\mathcal U}}
\def\cV{{\mathcal V}}      \def\cX{{\mathcal X}}
   \def\cZ{{\mathcal Z}}

      \def\bL{{\mathbf L}}
\def\bM{{\mathbf M}}      
      
\def\bS{{\mathbf S}}   \def\bT{{\mathbf T}}

\def\dim{{\rm dim\,}}
\def\wt{\widetilde}
\def\wh{\widehat}

\def\bL{{\mathbf L}}

\def\RE{{\rm Re\,}}
\def\IM{{\rm Im\,}}
\def\Ext{{\rm Ext\,}}
\def\ran{{\rm ran\,}}
\def\cran{{\rm \overline{ran}\,}}
\def\dom{{\rm dom\,}}
\def\cspan{{\rm \overline{span}\, }}
\def\uphar{{\upharpoonright\,}}
\def\Ext{{\rm Ext\,}}

\def\f{\varphi}

\DeclareMathOperator{\hplus}{\, \widehat + \,}

\newtheorem{theorem}{Theorem}[section]

\newtheorem{proposition}[theorem]{Proposition}
\newtheorem{corollary}[theorem]{Corollary}
\newtheorem{definition}[theorem]{Definition}
\newtheorem{remark}[theorem]{Remark}

\numberwithin{equation}{section}

\numberwithin{equation}{section}
\begin{document}

\title[Compressed  resolvents of $sc$-extensions with exit]
{Compressed resolvents of selfadjoint contractive extensions with exit and holomorphic operator-functions associated with them}

\author[Yury Arlinski\u{\i}]{Yu.M. Arlinski\u{\i}}
\address{Department of Mathematical Analysis  \\
East Ukrainian National University  \\
and
Department of Computer Mathematics and Mathematical Modeling\\
National Technical University "Kharkiv  Polytechnic Institute"\\
21, Frunze str.\\
Kharkiv, 61002, Ukraine}

\email{yury.arlinskii@gmail.com}

\author[Seppo Hassi]{S. Hassi}
\address{Department of Mathematics and Statistics  \\
University of Vaasa  \\
P.O. Box 700  \\
65101 Vaasa  \\
Finland} \email{sha@uwasa.fi}


\vskip 1truecm
\thispagestyle{empty}
\baselineskip=12pt

\begin{abstract}
Contractive selfadjoint extensions of a Hermitian contraction $B$ in
a Hilbert space $\sH$ with an exit in some larger Hilbert space
$\sH\oplus\cH$ are investigated. This leads to a new geometric
approach for characterizing analytic properties of holomorphic
operator-valued functions of Kre\u{\i}n-Ovcharenko type, a class of
functions whose study has been recently initiated by the authors.
Compressed resolvents of such exit space extensions are also
investigated leading to some new connections to transfer functions
of passive discrete-time systems and related classes of holomorphic
operator-valued functions.
\end{abstract}

\maketitle
\section{Introduction}

Let $S$ be a closed symmetric, possibly nondensely defined, linear
operator in a (complex separable) Hilbert space $\sH$. As is well
known, the operator $S$ admits selfadjoint extensions possibly in a
larger Hilbert space $\wt \sH=\sH\oplus\cH$ \cite{AkhGl},
\cite{Naj0}. Let $\wt A$ be such extension. Then there are two
compressed resolvents
$P_\sH(\wt A-\lambda I)^{-1}\uphar\sH$ and $P_\cH(\wt A -\lambda
I)^{-1}\uphar\cH.$  As is well known, the function $P_\sH(\wt
A-\lambda I)^{-1}\uphar\sH$ is called \textit{generalized resolvent}
of $S$. First results related to descriptions/parameterizations of
canonical and generalized resolvents of densely defined closed
symmetric operator with equal and finite deficiency indices, and
their applications to the moment and interpolation problems were
obtained by M.A.~Na\u{\i}mark \cite{Naj1,Naj2} and M.G.~Kre\u\i n
\cite{Kr44,Kr46,Kr1967}. Kre\u\i n's approach has been further
developed in M.G.~Kre\u\i n and H.~Langer \cite{KL1,KL2}, where
densely defined symmetric operators in a Pontryagin space setting
were considered. A.V.~Shtraus in \cite{shtraus1954} suggested
another approach for the investigation and parametrization of all
generalized resolvents of an arbitrary symmetric, not necessary
densely defined, operator. The Shtraus representation
\cite{shtraus1954} for $P_\sH(\wt A-\lambda I)^{-1}\uphar\sH$ takes
the form
\[
P_\sH(\wt A-\lambda I)^{-1}\uphar\sH=(\cA(\lambda)-\lambda
I)^{-1},\; \lambda\in\dC\setminus\dR_+,
\]
where $\cA(\lambda)$ is a holomorphic family of quasi-selfadjoint
extensions of $S$ $(S\subset \cA(\lambda)\subset S^*)$, $A(\lambda)$
is maximal dissipative for $\IM\lambda<0$, and maximal
anti-dissipative for $\IM\lambda>0.$ A recent survey on Shtraus
approach, its developments, and corresponding references can be
found in \cite{zagorod}. Extensions of symmetric linear relations
and their generalized resolvents have been studied in
\cite{codd1972,codd1974,DdeS,LT}. Furthermore, M.G.~Kre\u\i n and
I.E.~Ovcharenko \cite{KrO} and H.~Langer and B.~Textorius \cite{LT1}
obtained descriptions of all generalized resolvents of selfadjoint
contractive extensions and contractive extensions of dual pair of
contractions.

The main objective in this paper is to study compressed resolvents
$P_\sH(z\wt B-I)^{-1}\uphar\sH$ and $P_\cH(z\wt B-I)^{-1}\uphar\cH$
of selfadjoint contractive extensions ($sc$-extensions) $\wt B$
(with exit in some larger complex separable Hilbert space
$\sH\oplus\cH$) of a nondensely defined Hermitian contraction $B$ in
$\sH$ and investigate the interplay that occurs in certain
associated analytic operator functions. This investigation is
motivated by some further applications which involve boundary
triplets, boundary relations, and the corresponding Weyl functions
and Weyl families; cf. \cite{DM1,DM2,DHMS06}. In this paper some new
connections between compressed resolvents and transfer functions of
corresponding passive selfadjoint discrete-time systems are
established; see Theorems \ref{comrescontr}, \ref{realiz1} with a
further consequence established in Theorem \ref{repweyl11}. There
are also a couple of other new properties that complement some
well-known results established in \cite{Kr,KrO} and are related to
the shorted operators and selfadjoint contractive extensions; see
Theorems \ref{shorts11} and \ref{equshorts}. These results lead to a
new construction of special pairs of $sc$-extensions of $B$ without
exit by means of $sc$-extensions with exit with certain prescribed
geometric properties. The main result in this connection is
established in Theorem~\ref{construc}. The interest in studying such
special pairs of $sc$-extensions of $S$ possessing certain specific
geometric properties comes from the fact that they play a central
role in characterizing analytic properties of Kre\u{\i}n-Ovcharenko
type holomorphic operator functions which originally appeared in
\cite{KrO} and whose systematic study was initiated in \cite{AHS}.

\section{Preliminaries}

\subsection{Linear fractional transformation of sectorial operators
and linear relations} \label{FLTr}

On the set of all linear relations (l.r.) in a Hilbert space $\sH$
define the linear fractional transformation (the Cayley transform)
\begin{equation}\label{CALTR}
 \cC(\bS)=\bT=\left\{\left<x+x',x-x'\right>:
\left<x,x'\right>\in\bS\right\}.
\end{equation}
Clearly, $\cC(\cC(\bS))=\bS$. Let $\bS$ be an accretive l.r. in
$\sH$, i.e., $\RE(x',x)\ge 0$ for all $\left<x,x'\right>\in\bS$; see
\cite{Ka,RoBe}. Then it follows from the identity
\[
 \|x+x'\|^2-\|x-x'\|^2=4\RE(x',x)
\]
that $\bT$ is the graph of a contraction $T$ in $\sH$, $||T||\le 1$,
and $\dom T =\dom \bT$ is a subspace in $\sH$. Conversely, if $T$ is
a contraction in $\sH$ defined on a subspace $\dom T\subseteq \sH$,
then
\[
\bS=\left\{\left<(I+T)h,(I-T)h\right>,\; h\in\dom T\right\}
\]
is an accretive l.r. in $\sH$. The transformation $\cC$ can be
rewritten in operator form as follows
\[
\cC(\bS)= T=-I +2(I+\bS)^{-1},\; \bS=(I-T)(I+T)^{-1}=-I+2(I+T)^{-1}.
\]
The following properties are clear from the above formulas:
\begin{itemize}
\item $\bS$ is the graph of an accretive operator if and only
if $\ker(I_H+T)=\{0\},$
\item $\bS$ is $m$-accretive if and only if $\dom T=H$,
\item $\bS$ is nonnegative selfadjoint relation if and only if $T$ is a selfadjoint
contraction.
\end{itemize}
In the sequel we will denote by $\bL(\sH_1,\sH_2)$ the set of all
linear bounded operators acting from $\sH_1$ into $\sH_2$ and by
$\bL(\sH)$ the Banach algebra $\bL(\sH,\sH)$.

Recall that for a contraction $T\in \bL(\sH_1,\sH_2)$ the
nonnegative square root $D_T =(I - T^*T)^{1/2}$ is called the defect
operator of $T$ and ${\sD}_T$ (the so-called defect subspace)
denotes the closure of the range ${\ran}D_T$. For the defect
operators one has the well-known commutation relation $TD_T =
D_{T^*}T$. Since
\[
 \begin{bmatrix} T & D_{T^*} \end{bmatrix}\begin{bmatrix} T & D_{T^*}
 \end{bmatrix}^*
 =TT^*+D_{T^*}^2=I_2,
\]
one has $\ran T+ \ran D_{T^*}=\sH_2$. In general this sum is not
direct: one has
\begin{equation}\label{DTran1}
 \ran T\cap \ran D_{T^*}=\ran TD_T=\ran D_{T^*}T,
\end{equation}
as can be checked directly. It is also easily seen that
\begin{equation}\label{DTran2}
 T(\ker D_T)=\ker D_{T^*}, \quad T^*(\ker D_{T^*})=\ker D_{T}.
\end{equation}
Hence, $\ker D_T=\{0\}$ if and only if $\ker D_{T^*}=\{0\}$.

\begin{definition}
\label{CAS} \cite{Arl1987}. Let $\alpha\in (0,\pi/2)$ and let $A$ be
a linear operator in the Hilbert space $H$ defined on a subspace
$\dom A$. If
\begin{equation}
\label{CA1} ||A\sin\alpha\pm i\cos\alpha I_H||\le 1,
\end{equation}
then in the case $\dom A=H$ we say that $A$ belongs to the class
$C_H(\alpha)$, and in the case $\dom A\ne H$ we say that $A$ is
$C_H(\alpha)$-suboperator.
\end{definition}
The condition \eqref{CA1} is equivalent to
\begin{equation}
\label{CA2} 2|\IM (Af,f)|\le\tan\alpha(||f||^2-||Af||^2),\; f\in\dom
A.
\end{equation}
Therefore, $C_H(\alpha)$-suboperator is a contraction. Due to
\eqref{CA2} it is natural to consider Hermitian (selfadjoint)
contractions in $H$ as $C_H(0)$-suboperators (operators of the class
$C_H(0)$, respectively). In view of \eqref{CA2} one can write
\[
C_H(0)=\underset{\alpha\in(0,\pi/2)}{\bigcap}C_H(\alpha).
\]
Analogously, the convex hull
$C(\alpha)=\left\{z\in\dC:|z\sin\alpha\pm i\cos\alpha|<1\right\}$ in
the complex plane is denoted by $C(\alpha)$. If $\alpha=0$, then the
above intersection equals $C(0)=[-1,1].$ Notice that the linear
fractional transformation \eqref{CALTR} establishes a one-to-one
correspondence between $\alpha$-sectorial ($m-\alpha$-sectorial)
l.r. (as defined in \cite{Ka,RoBe}) in $H$ and
$C_H(\alpha)$-suboperators (operators of the class $C_H(\alpha)$,
respectively). In addition, $T\in C_H(\alpha)$ if and only if the
operator $(I-T^*)(I+T)$ is a sectorial operator with the vertex at
the origin and the semiangle $\alpha$; see \cite{Arl1991}. Denote
\[
 \wt C_H:=\underset{\alpha\in[0,\pi/2)}{\bigcup}C_H(\alpha).
\]
Properties of operators of the class $\wt C_H$ were studied in
\cite{Arl1987, Arl1991}. In \cite{Arl1987} it was proved that if
$T\in\wt C_H$, then
\begin{enumerate}
\item $
\ran(D_{T^n})=\ran(D_{T^{*n}})=D_{T_R}$ for all natural numbers $n$,
where $T_R=(T+T^*)/2$ is the real part of $T$,
\item
  the subspace $\sD_T$ reduces the operator $T$, and,
moreover, the operator $T\uphar\ker(D_T)$ is a selfadjoint and
unitary, and $T\uphar\sD_{T}$ is a completely nonunitary contraction
of the class $C_{00}$ \cite{SF}, i.e., $\lim\limits_{n\to\infty}T^n
f=\lim\limits_{n\to\infty}T^{*n} f=0$ for all $f\in\sD_{T}.$
\end{enumerate}
Let $T\in\wt C_H$. Then, clearly, the operators $I_H\pm T$ are
m-sectorial (bounded) operators. It follows that
$
I_H+T=(I_H+T_R)^{1/2}(I+iG)(I_H+T_R)^{1/2}, $ where $T_R=(T+T^*)/2$
is the real part of $T$, $G$ is a bounded selfadjoint operator in
the subspace $\cran(I_H+T_R)^{1/2}$, and $I$ is the identity
operator in $\cran(I_H+T_R)^{1/2}$. Let
\[
\bM=-I+2(I_H+T)^{-1}=\left\{\left\{(I_H+T)f,(I_H-T)f\right\},\; f\in
H\right\}.
\]
Then $\bM$ is m-sectorial linear relation, $\dom \bM=\ran(I_H+T)$.
The closed sectorial form $\bM[u,v]$ generated by $\bM$ can be
described now explicitly.

\begin{proposition}\label{clfrm}
The closed sectorial form associated with m-sectorial linear
relation $\bM$ is given by
\begin{equation}
\label{viraz2}
 \bM[u,v]=-(u,v)+2\left((I+iG)^{-1}(I_H+T_R)^{-1/2}u,(I_H+T_R)^{-1/2}v\right),
\end{equation}
for all $u,v\in\cD[\bM]=\ran(I_H+T_R)^{1/2}$.
\end{proposition}
\begin{proof}
Let $g=(I_H+T)f$, $g'=(I_H-T)f$. Then $\{g,g'\}\in\bM$. With $u=g$
one gets
\[
\begin{array}{ll}
(\bM u,u)& =(g',g)=((I_H-T)f,(I_H+T)f)\\
&=-((I_H+T)f,(I_H+T)f)+2(f,(I_H+T)f)\\
&=-||u||^2+2((I_H+T)^{-1}u,u)\\
&=-||u||^2+2((I_H+T_R)^{-1/2}(I+iG)^{-1}(I_H+T_R)^{-1/2}u,u)\\
&=-||u||^2+2((I+iG)^{-1}(I_H+T_R)^{-1/2}u,(I_H+T_R)^{-1/2}u).
\end{array}
\]
It follows that the righthand side of \eqref{viraz2} coincides with
$\bM[u,v]$ for $u,v\in\dom\bM$.

Let $H_0=\cran(I_H+T)$. Then $\ran(I_H+T_R)^{1/2}$ is dense in
$H_0$. Denote
\[
\tau[u,v]=-||u||^2+2\left((I+iG)^{-1}(I_H+T_R)^{-1/2}u,(I_H+T_R)^{-1/2}v\right),
\]
with $u,v \in\ran(I_H+T_R)^{1/2}.$ Clearly, the form  $\tau$ is
closed and sectorial (with a vertex at the point $-1$ at least). Let
$u=(I_H+T_R)^{1/2}h$, $h\in H_0$, and choose a sequence
$\{h_n\}\subset H_0$ such that
$\lim\limits_{n\to\infty}(I_H+T_R)^{1/2}h_n=(I+iG)^{-1}h\in H_0$.
Then $\f_n=(I_H+T)h_n\in \dom \bM$ and
$\lim\limits_{n\to\infty}\f_n=(I+T_R)^{1/2}h=u$. Moreover,
\[
\begin{array}{l}
\tau[u-\f_n]\\
\quad =-||u-\f_n||^2 + 2\left((I+iG)^{-1}(I_H+T_R)^{-1/2}(u-\f_n),(I_H+T_R)^{-1/2}(u-\f_n)\right)\\
\quad =-||u-\f_n||^2 + 2\left((I+iG)^{-1}h-(I_H+T_R)^{1/2}h_n,
h-(I+iG)(I_H+T_R)^{1/2}h_n)\right).
\end{array}
\]
Hence $\lim\limits_{n\to\infty}\tau[u-\f_n]=0$. This shows that the
form $\tau$ is the closure of the form $(\bM\cdot,\cdot)$ and this
completes the proof.
\end{proof}

\subsection{Passive discrete-time systems and their
transfer functions}  \label{secS}

Let $\sM,\sN$, and $\sH$ be separable Hilbert spaces. A linear
system $\tau=\left\{\begin{bmatrix} D&C \cr
B&A\end{bmatrix};\sM,\sN, \sH\right\}$ with bounded linear operators
$A$, $B$, $C$, $D$  of the form
\[
 \left\{
 \begin{array}{l}
    \sigma_k=Ch_k+D\xi_k,\\
    h_{k+1}=Ah_k+B\xi_k
\end{array}
\right. \qquad k\in\dN_0,
\]
where $\{\xi_k\}\subset \sM$, $\{\sigma_k\}\subset \sN$,
$\{h_k\}\subset \sH$ is called a \textit{discrete time-invariant
system}. The Hilbert spaces $\sM$ and $\sN$ are called the input and
the output spaces, respectively, and the Hilbert space $\sH$ is
called the state space. Associated with $\tau$ is the block operator
\[
U=\begin{bmatrix} D&C \cr B&A\end{bmatrix} :
\begin{array}{l} \sM \\\oplus\\ \sH \end{array} \to
\begin{array}{l} \sN \\\oplus\\ \sH \end{array}.
\]
If $U$ is contractive, then the corresponding discrete-time system
is said to be \textit{passive} \cite{A}. If $U$ is unitary, then the
system is called conservative. The \textit{transfer function}
\[
\Theta_\tau(\lambda):=D+z C(I_{\sH}-z A)^{-1}B, \quad z \in \dD,
\]
of a passive system $\tau$ belongs to the \textit{Schur class} ${\bf
S}(\sM,\sN)$ \cite{A}. Recall that the Schur class ${\bf
S}(\sM,\sN)$ is the set of all holomorphic and contractive
$\bL(\sM,\sN)$-valued functions on the unit disk
$\dD=\{z\in\dC:|z|<1\}$.

Define the following subsets of the complex plane
\[
\begin{array}{l}
\Pi_{+}(\alpha):=\{z\in \dC:|z\sin\alpha+ i\cos\alpha|< 1\},\;
    \Pi_{-}(\alpha):=\{z\in \dC:|z\sin\alpha-
    i\cos\alpha|<1\},\\
\Pi(\alpha):=\Pi_{+}(\alpha)\cup\Pi_-(\alpha).
\end{array}
\]
Then, in particular
$
\Pi(0)=\dC\setminus\left((-\infty,1]\cup[1,+\infty)\right) $ and
$C(\alpha)=\Pi_+(\alpha)\cap\Pi_-(\alpha)$.
\begin{theorem}
\label{calsys}{\rm{\cite{Arl1991}}}. Suppose that $\sN=\sM$ and that
the operator
\[
U=\begin{bmatrix} D&C \cr B&A\end{bmatrix} :
\begin{array}{l} \sN \\\oplus\\ \sH \end{array} \to
\begin{array}{l} \sN \\\oplus\\ \sH \end{array}.
\]
belongs to class $C_{\sN\oplus\sH}(\alpha)$ for some $\alpha\in
[0,\pi/2)$. Then the function $\Theta_\tau$ possesses the following
properties:
\begin{enumerate}
\item
$\Theta_\tau$ is holomorphic in $\Pi(\alpha)$;
\item there exist strong non-tangential limits $\Theta_\tau(\pm 1)$
and $\Theta_\tau(\pm 1)\in C_\sN(\alpha)$;
\item the implications
\[
\begin{array}{l}
z\in\Pi_+(\alpha)\Longrightarrow
||\Theta_\tau(z)\sin\alpha+i\cos\alpha \,I_\sN||\le 1,\\
z\in\Pi_-(\alpha)\Longrightarrow
||\Theta_\tau(z)\sin\alpha-i\cos\alpha \,I_\sN||\le 1
\end{array}
\]
are valid. Therefore, $ z\in C(\beta)\Longrightarrow
\Theta_\tau(z)\in C_\sN(\beta)$ for each $\beta\in[\alpha,\pi/2).$
\end{enumerate}
\end{theorem}

A particular case is \textit{self-adjoint passive system}, i.e., the
case when $\alpha=0$ $\iff$ the operator $U=\begin{bmatrix} D&C \cr
B&A\end{bmatrix}$ is a self-adjoint contraction in $\sN\oplus\sH$.

A more general class of passive systems is formed by \textit{passive
quasi-selfadjoint systems} ($pqs$-systems for short). The passive
system
$$\tau=\left\{\begin{bmatrix} D&C \cr B&A\end{bmatrix};\sN,\sN,\sH\right\}$$
is called a $pqs$-system if the operator
$
U=\begin{bmatrix} D&C \cr B&A\end{bmatrix}
$
is a \textit{quasi-selfadjoint contraction} ($qsc$-operator for
short), i.e., $U$ is a contraction and $\ran (U-U^*)\subset
\sN\times\{0\}$, cf. \cite{AHS2}. This last condition alone is
equivalent to $A=A^*$ and $C=B^*$; for contractivity of $U$ see
Theorem~\ref{qscex} below. If $\tau$ is a $pqs$-system, then the
transfer function  of $\tau$ takes the form
\[
\Theta_\tau(z)=W(z)+D,
\]
where the function $W(z)$ belongs to the class $\mathbf{N}(\sN) $ of
Herglotz-Nevanlinna functions and it is defined on the domain
$\Ext\{(-\infty,-1]\cup [1,\infty)\}$. The class ${\bf S}^{qs}(\sN)$
is the class of all transfer functions of $pqs$-systems
$\tau=\{U;\sN,\sN,\sH\}$. A complete description of the class  ${\bf
S}^{qs}(\sN)$ is given in \cite{AHS3}. Denote by ${\bf S}^{s}(\sN)$
the subset of Herglotz-Nevanlinna functions from the class of ${\bf
S}^{qs}(\sN)$.  Clearly,
\[
\Theta(z)\in {\bf S}^{s}(\sN)\iff\left\{\begin{array}{l} \Theta(z)\in {\bf S}^{qs}(\sN),\\
\Theta(0)=\Theta^*(0)\end{array}\right. .
\]
The following equivalent statements for $\bL(\sN)$-valued
Herglotz-Nevanlinna function $\Theta$, holomorphic in
$\dC\setminus\left\{(-\infty,-1]\cup [1,\infty)\right\}$, can be
derived with the aid of the integral representation of $\Theta$; see
also \cite[Theorem 4.2]{KrO}:
\begin{enumerate}
\item $\Theta \in {\bf S}^{s}(\sN)$;
\item $\Theta(x)$ is selfadjoint contraction for each $x\in(-1,1)$;
\item $\Theta$ is the transfer function of a passive selfadjoint discrete-time system
$$\tau=\left\{\begin{bmatrix} D&B \cr B^*&A\end{bmatrix};\sN,\sN,\sH\right\}.$$
\end{enumerate}
\subsection{The Schur-Frobenius formula for the resolvent.}
Let
\[
\cU=\begin{bmatrix} D&C \cr B&A\end{bmatrix} :
\begin{array}{l} \sM \\\oplus\\ \sH \end{array} \to
\begin{array}{l} \sM \\\oplus\\ \sH \end{array}.
\]
be a bounded block operator. Then an applications of the
Schur-Frobenius formula gives the following formula for the
resolvent $R_\cU(\lambda)=(\cU-\lambda I)^{-1}$ of $\cU$:
\begin{equation}
\label{Sh-Fr1}
\begin{array}{l}
R_\cU(\lambda)=
\begin{bmatrix}-V^{-1}(\lambda)&V^{-1}(\lambda)CR_A(\lambda)\cr
R_A(\lambda)BV^{-1}(\lambda)&R_A(\lambda)\left(I_\cH-BV^{-1}(\lambda)CR_A(\lambda)\right)
\end{bmatrix},
\quad \lambda\in\rho(\cU)\cap\rho(A),
\end{array}
\end{equation}
where
\begin{equation}
\label{VT}V(\lambda):=\lambda I_\sM-D+CR_A(\lambda)B,\;
\lambda\in\rho(A).
\end{equation}
Moreover, $\lambda\in\rho(\cU)\cap\rho(A)\iff
V^{-1}(\lambda)\in\bL(\sM)$. In particular, \eqref{Sh-Fr1} and
\eqref{VT} imply
\begin{equation}
\label{inv11} \left(P_\sM R_U(\lambda)\uphar\sM\right)^{-1}
=D-CR_A(\lambda)B-\lambda I_\sM.
\end{equation}

\subsection{Kre\u\i n shorted operators}

For every bounded nonnegative operator $\cS$ in the Hilbert space
$\cH$ and every subspace $\cK\subset \cH$ M.G.~Kre\u{\i}n \cite{Kr}
defined the operator $\cS_{\cK}$ by the relation
\[
 \cS_{\cK}=\max\left\{\,\cZ\in \bL(\cH):\,
    0\le \cZ\le \cS, \, {\ran}\cZ\subseteq{\cK}\,\right\}.
\]
An equivalent description is
\begin{equation}\label{Sh1}
 \left(\cS_{\cK}f, f\right)=\inf\limits_{\f\in \cK^\perp}\left\{\left(\cS(f + \varphi),f +
 \varphi\right)\right\},
\quad  f\in\cH,
\end{equation}
where $\cK^\perp:=\cH\ominus{\cK}$. The properties of $\cS_{\cK}$,
have been studied by M.G.~Kre\u{\i}n and by other authors (see
\cite{ARL1} and references therein): in \cite{And,AT} $\cS_{\cK}$ is
called a \textit{shorted operator}. The following representation of
$\cS_{\cK}$ was also established in \cite{Kr}:
\[
 \cS_{\cK}=\cS^{1/2}P_{\Omega}\cS^{1/2},
\]
 where $P_{\Omega}$ is the orthogonal projection in
$\cH$ onto
$
 \Omega=\{\,f\in \cran \cS:\,\cS^{1/2}f\in {\cK}\,\}=\cran \cS\ominus
\cS^{1/2}\cK^\perp. $ Moreover, it was shown in \cite{Kr} that
\begin{equation}\label{Sh2}
 {\ran}\cS_{\cK}^{1/2}={\ran}(\cS^{1/2}P_\Omega)={\cK}\cap{\ran}\cS^{1/2}.
\end{equation}
Hence,
\begin{equation}
\label{nol}
  \cS_{\cK}=0 \iff  \ran \cS^{1/2}\cap \cK=\{0\}.
\end{equation}

As a bounded selfadjoint operator $\cS$ admits the block operator
representation
\[
\cS=\begin{bmatrix}\cS_{11}&\cS_{12}\cr \cS^*_{12}&\cS_{22}
\end{bmatrix}:\begin{array}{l}\cK\\\oplus\\\cK^\perp \end{array}\to
\begin{array}{l}\cK\\\oplus\\\cK^\perp \end{array}.
\]
It is well known (see \cite{HMS,KrO,S59}) that the operator $\cS$ is
nonnegative if and only if
\[
\cS_{22}\ge 0,\; \ran \cS^*_{12}\subset\ran \cS^{1/2}_{22},\,\;
\cS_{11}\ge
\left(\cS^{-1/2}_{22}\cS^*_{12}\right)^*\left(\cS^{-1/2}_{22}\cS^*_{12}\right)
\]
and the operator $\cS_\cK$ can be expressed in the block operator
form
\begin{equation}
\label{shormat1}
\cS_\cK=\begin{bmatrix}\cS_{11}-\left(\cS^{-1/2}_{22}\cS^*_{12}\right)^*\left(\cS^{-1/2}_{22}\cS^*_{12}\right)&0\cr
0&0\end{bmatrix},
\end{equation}
where $\cS^{-1/2}_{22}$ is the Moore-Penrose pseudo-inverse of
$\cS_{22}$. If $\cS^{-1}_{22}\in\bL(\cK^\perp)$ then
\[
\cS_\cK=\begin{bmatrix}\cS_{11}-\cS_{12}\cS^{-1}_{22}\cS^*_{12}&0\cr
0&0\end{bmatrix}
\]
and $\cS_{11}-\cS_{12}\cS^{-1}_{22}\cS^*_{12}$ is called a
\textit{Schur complement} of $\cS$. From \eqref{shormat1} it follows
that
\[
\cS_\cK=0\iff \ran \cS^*_{12}\subset\ran
\cS^{1/2}_{22}\quad\mbox{and}\quad
\cS_{11}=\left(\cS^{-1/2}_{22}\cS^*_{12}\right)^*\left(\cS^{-1/2}_{22}\cS^*_{12}\right).
\]

\subsection{Selfadjoint and quasi-selfadjoint contractive extensions of a nondensely defined Hermitian contraction}
\label{QSCEX} Let $B$ be a closed nondensely defined Hermitian
contraction in the Hilbert space $\sH$. Denote
\[
\sH_0:=\dom B,\; \sN:=\sH\ominus\sH_0.
\]
A description of all selfadjoint contractive extensions
($sc$-extensions \cite{KrO}) of $B$ in $\sH$ was given by
M.G.~Kre\u{\i}n \cite{Kr}. In fact, he showed that all
$sc$-extensions of $B$ form an operator interval $[B_\mu, B_M]$,
where the extensions $B_\mu$ and $B_M$ can be characterized by
\begin{equation}
\label{extr} \left(I+B_\mu\right)_{\sN}=0, \quad
\left(I-B_M\right)_{\sN}=0,
\end{equation}
respectively. The operator $B$ admits a unique $sc$-extension if and
only if
\[
\sup\limits_{\f\in\dom B
}\cfrac{|(B\f,h)|^2}{||\f||^2-||B\f||^2}=\infty
\]
for all $h\in\sN\setminus\{0\}.$

The operator interval $[B_\mu, B_M]$ can be described as follows
(cf. \cite{Kr,KrO}):
\begin{equation}\label{param1}
 \wh{B} = (B_M + B_\mu)/2 + (B_M - B_\mu)^{1/2} Y (B_M - B_\mu)^{1/2} /2,
\end{equation}
where $Y=Y^*$ is a contraction in the subspace $\cran(B_M -
B_{\mu})\subseteq{\sN}$. It follows from \eqref{extr}, for instance,
that for every $sc$-extension $\wh B$ of $B$ the following
identities hold:
\begin{equation}
\label{rn1}
 (I-\wh B)_{\sN}=B_M-\wh B,
\quad
 (I+\wh B)_{\sN} =\wh B-B_\mu,
\end{equation}
cf. \cite{Kr}. Hence, according to \eqref{Sh2}
\[
\begin{array}{l}
 \ran(I-\wh B)^{1/2}\cap{\sN}=\ran(B_M-\wh B)^{1/2},\\
 \ran(I+\wh B)^{1/2}\cap{\sN}=\ran(\wh B-B_\mu)^{1/2}.
\end{array}
\]
Let $P_{\sH_0}$ and $P_{\sN}$ be the orthogonal projections in $\sH$
onto $\sH_0$ and $\sN$, respectively. Then the operator
$B_0=P_{\sH_0} B$ is contractive and self-adjoint in the subspace
$\sH_0$. Let $D_{B_0}= (I-B^2_0)^{1/2}$ be the defect operator
determined by $B_0$. The operator $B_{21}=P_{\sN}B$ is also
contractive. Moreover, it follows from $B^*B\le I$ that
$B_{21}^*B_{21}\le D_{B_0}^2$. Therefore, the identity
\[
 K_0 D_{B_0}f=P_{\sN}Bf, \; f\in \dom B=\sH_0,
\]
defines a contractive operator $K_0$ from
$\sD_{B_0}:=\cran(D_{B_0})$ into $\sN$, cf.~\cite{Doug,FW}. This
gives the following decomposition for the Hermitian contraction $B$
\begin{equation}
\label{HC}
 B=B_0+K_0 D_{B_0}=\begin{bmatrix} B_0 \cr K_0D_{B_0}
 \end{bmatrix}:\sH_0\to\sH.
\end{equation}
An extension $\wh B$ of $B$ in $\sH$ is called
\textit{quasi-selfadjoint} if also $\wh B^*$ is an extension of $B$
and $\wh B$ is said to be a \textit{quasi-selfadjoint contractive}
extension of $B$ ($qsc$-extension for short) if $\dom\wh B =\sH$,
$||\wh B||\le 1$, and $\ker (\wh B-\wh B^*)\supseteq\dom B=\sH_0$;
cf. \cite{ArTs00,ArTs0}.

For a proof of the following result and some history behind the
well-known formula therein; see
\cite[Theorem~9.2.3]{ArlBelTsek2011}, \cite[Theorem~4.1]{AHS2007},
\cite[Corollary~3.5]{HMS}.

\begin{theorem}
\label{qscex} Let $B$ be a Hermitian contraction in
$\sH=\sH_0\oplus\sN$ with $\dom B=\sH_0$ and decompose $B$ as in
\eqref{HC}. Then the formula
\begin{equation}
\label{MCE}
 \wh B=\begin{bmatrix} B_0& D_{B_0}K^*_0 \cr
           K_0D_{B_0}&  -K_0B_0K^*_0+D_{K^*_0}XD_{K^*_0}\end{bmatrix}:
\begin{array}{l}\sH_0\\\oplus\\\sN  \end{array}
 \to
\begin{array}{l}\sH_0\\\oplus\\\sN  \end{array}
\end{equation}
gives a one-to-one correspondence between all $qsc$-extensions $\wh
B$ of the Hermitian contraction $B=B_0+K_0D_{B_0}$ and all
contractions $X$ in the subspace
$\sD_{K_0^*}:=\cran(D_{K_0^*})\subseteq \sN$. Furthermore, the
following statements hold:
\begin{enumerate}
\def\labelenumi{\rm (\roman{enumi})}

\item $B$ has a unique $sc$-extension if and only if
$K^*_0$ is an isometry ($\sD_{K^*_0}=\{0\}$);

\item if $\sD_{K^*_0}\ne\{0\}$, then the following equivalences hold
\[
\ker D_{K_0^*}=\{0\}\iff\ker(B_M-B_\mu)=\sH_0\iff S_F\cap S_K=S;
\]

\item if $\sD_{K^*_0}\ne\{0\}$, then the following equivalences hold
\[
\ran D_{K_0^*}=\sN \iff \ran(B_M-B_\mu)=\sN\iff S_F\hplus S_K=S^*.
\]
\end{enumerate}

Moreover, $\wh B\in C_\sH(\alpha)$, $\alpha\in[0,\pi/2),$ if and
only if $X\in C_{\sD_{K^*_0}}(\alpha).$
\end{theorem}
From  \eqref{MCE} it follows that
\begin{equation}
\label{ENDS} B_\mu=\begin{bmatrix} B_0& D_{B_0}K^*_0 \cr
K_0D_{B_0}&-K_0B_0K^*_0-D^2_{K^*_0}\end{bmatrix}, \quad
B_M=\begin{bmatrix} B_0& D_{B_0}K^*_0 \cr
K_0D_{B_0}&-K_0B_0K^*_0+D^2_{K^*_0}\end{bmatrix}
\end{equation}
with $X=-I\uphar\sD_{K^*_0}$ and $X=I\uphar\sD_{K^*_0}$,
respectively. From \eqref{ENDS} it is seen that
\[
\frac{B_\mu+B_M}{2}=\begin{bmatrix} B_0& D_{B_0}K^*_0
   \cr K_0D_{B_0}&-K_0B_0K^*_0\end{bmatrix},
\quad
  \frac{B_M-B_\mu}{2}=\begin{bmatrix} 0&0\cr 0&D^2_{K^*_0}\end{bmatrix}.
\]
Finally, we mention the following implications
\begin{equation}\label{XtoB}
\begin{array}{l}
X\in\bL(\sD_{K^*_0}),\; ||X\sin\alpha+i\cos\alpha ||\le
1\Longrightarrow ||\wh B\sin\alpha+i\cos\alpha||\le 1,\\
X\in\bL(\sD_{K^*_0}),\; ||X\sin\alpha-i\cos\alpha ||\le
1\Longrightarrow||\wh B\sin\alpha-i\cos\alpha ||\le 1,
\end{array}
\end{equation}
where $\wh B$ is given by \eqref{MCE}.

\begin{remark}
\label{parr} Let $X$ be a selfadjoint contraction in the Hilbert
space $\cH_1\oplus\cH_2$. From Theorem \ref{qscex} one can derive
the following two block representations for $X$:
\begin{multline*}
 X =\begin{bmatrix} X_{11}& D_{X_{11}}L^*\cr LD_{X_{11}}&
     -LX_{11}L^*+D_{L^*}YD_{L^*}\end{bmatrix}\\
  =\begin{bmatrix} -UX_{22}U^*+D_{U^*}VD_{U^*}&
UD_{X_{22}}\cr
D_{X_{22}}U^*&X_{22}\end{bmatrix}:\begin{array}{l}\cH_1\\\oplus\\\cH_2\end{array}\to
\begin{array}{l}\cH_1\\\oplus\\\cH_2\end{array},
\end{multline*}
where $L\in\bL(\sD_{X_{11}},\cH_2)$ and
$U\in\bL(\sD_{X_{22}},\cH_1)$ are contractions and
$Y\in\bL(\sD_{L^*})$ and $V\in\bL(\sD_{U^*})$ are selfadjoint
contractions. From \eqref{extr}, \eqref{rn1}, and \eqref{ENDS} we
get
\begin{multline*}
(I+X)_{\cH_2}=D_{L^*}(I+Y)D_{L^*}P_{\cH_2},\;
(I-X)_{\cH_2}=D_{L^*}(I-Y)D_{L^*}P_{\cH_2},\\
(I+X)_{\cH_1}=D_{U^*}(I+V)D_{U^*}P_{\cH_1},\;(I-X)_{\cH_1}=D_{U^*}(I-V)D_{U^*}P_{\cH_1},\\
(I-X)_{\cH_ 1}=(I+X)_{\cH_1}=0 \iff UU^*=I_{\cH_1}\Longrightarrow
X=\begin{bmatrix} -UX_{11}U^*& UD_{X_{22}}\cr
D_{X_{22}}U^*&X_{22}\end{bmatrix},
\end{multline*}
\begin{multline*}
(I-X)_{\cH_2}=(I+X)_{\cH_2}=0 \iff LL^*=I_{\cH_2}\Longrightarrow
X=\begin{bmatrix} X_{11}& D_{X_{11}}L^*\cr LD_{X_{11}}&
-LX_{11}L^*\end{bmatrix}.
\end{multline*}
In addition, for the defect operators the following identities hold
(cf. \cite[Theorem~4.1]{AHS2007}):
\[
\begin{array}{rcl}
 \left\|D_X\begin{pmatrix}h_1\cr h_2\end{pmatrix}\right\|^2
 &=&\left\|D_L\left(D_{X_{11}}h_1-X_{11}L^* h_2\right)
 -L^*YD_{L^*}h_2\right\|^2+\left\|D_YD_{L^*}h_2\right\|^2 \\
 &=&\left\|D_U\left(D_{X_{22}}h_2-X_{22}U^*h_1\right)
 -U^*VD_{U^*}h_1\right\|^2+\left\|D_VD_{U^*}h_1\right\|^2.
\end{array}
\]
\end{remark}

\subsection{Special pairs of selfadjoint contractive extensions and corresponding $Q$- functions}

The so-called $Q_\mu$ and $Q_M$-functions of a Hermitian contraction
$B$ of the form
\[
\begin{array}{l}
Q_\mu(\xi)=\left(I_{\sN}+(B_M- B_\mu)^{1/2}\left(
 B_\mu-\xi I_\sH\right)^{-1}(B_M- B_\mu)^{1/2}\right)\uphar\sN,\\
 Q_M(\xi)=\left(-I_{\sN}+(B_M- B_\mu)^{1/2}\left(
 B_M-\xi I_\sH\right)^{-1}(B_M- B_\mu)^{1/2}\right)\uphar\sN
,\;\xi\in\dC\setminus[-1,1],
 \end{array}
\]
were introduced and studied in \cite{KrO}. These functions belong to
the Herglotz-Nevanlinna class and they are connected to each other
via
\[
 Q_\mu(\xi)Q_M(\xi)=-I_\sN, \quad \xi\in\dC\setminus[-1,1].
\]
They possess the following further properties:
\[
\begin{array}{l}
 s-\lim\limits_{\xi\to\infty} Q_\mu(\xi)=I;\,\,
 \lim\limits_{\xi\uparrow-1}(Q_\mu(\xi)h,h)=+\infty \forall
 h\in\sN\setminus\{0\};\,\,
  s-\lim\limits_{\xi\downarrow 1} Q_\mu(\xi)=0;\\
  s-\lim\limits_{\xi\to\infty} { Q}_M(\xi)=-I_\sN;\,\,
     \lim\limits_{\xi\downarrow1}({Q}_M(\xi)h,h)=-\infty
   \forall  h\in\sN\setminus\{0\};\,\,
  s-\lim\limits_{\xi\uparrow -1} {Q}_M(\xi)=0.
\end{array}
\]
The following resolvent formula has been established in \cite{KrO}.

\begin{theorem}
\label{kror} Let $C=B_M-B_\mu$. The formula
\[
\wt R_\xi=(B_\mu-\xi I)^{-1} -(B_\mu-\xi
I)^{-1}C^{1/2}K(\xi)\left(I+(Q_\mu(\xi)-I)K(\xi)\right)^{-1}C^{1/2}(B_\mu-\xi
I)^{-1}
\]
gives a bijective correspondence between the generalized resolvents
$\wt R_\xi=P_\sH(\wt B-\xi I)^{-1}\uphar\sH$ of $sc$-extensions $\wt
B$ of $B$ with exit and the $\bL(\sN)$-valued operator functions
$K(\xi)$ holomorphic on $\Ext[-1,1]$ and possessing the following
two further properties:

1) $-K(\xi)$ is a Herglotz-Nevanlinna function,

2) $K(\xi)$ is a nonnegative selfadjoint contraction for every
$\xi\in\dR\setminus [-1,1]$.

\noindent Here canonical resolvents correspond to constant functions
$K(\xi)=K$ and vice versa.
\end{theorem}

A further study of functions of Kre\u{\i}n-Ovcharenko type was
initiated in \cite{AHS}. Given an arbitrary pair $\{\wh B_0,\wh
B_1\}$ of $sc$-extensions of $B$ in $\sH$ satisfying the condition
$\wh B_0\le \wh B_1$, define a pair of Herglotz-Nevanlinna functions
via
\begin{equation}
\label{q1} {\wh Q}_0(\xi)=\left[(\wh B_1-\wh B_0)^{1/2}(\wh B_0-\xi
I)^{-1}(\wh B_1-\wh B_0)^{1/2}+I\right] \uphar{\sN},
\end{equation}
\begin{equation}
\label{q2} {\wh Q}_1(\xi)=\left[(\wh B_1-\wh B_0)^{1/2}(\wh B_1-\xi
I)^{-1}(\wh B_1-\wh B_0)^{1/2}-I\right]\uphar{\sN},\quad \xi\in
\Ext[-1,1].
\end{equation}
It is easy to verify that ${\wh Q}_0(\xi){\wh Q}_1(\xi)={\wh
Q}_1(\xi){\wh Q}_0(\xi)=-I_{\sN}$, $\xi\in \Ext[-1,1].$ Now proceed
by introducing the classes of Kre\u{\i}n-Ovcharenko type
Herglotz-Nevanlinna functions.

\begin{definition}
\label{def1} {\rm \cite{AHS}}. Let $\sN$ be a Hilbert space. An
$\bL(\sN)$-valued function $\wh Q(\xi)$ is said to belong to the
subclass $\sS_\mu(\sN)$ (respectively, $\sS_M(\sN))$ of
Herglotz-Nevanlinna operator functions if it is holomorphic on
$\Ext[-1,1]$ and, in addition, has the following properties:
\begin{enumerate}
\item[1)] $s-\lim\limits_{\xi\to\infty} {\wh Q}(\xi)=I$ (respectively, $s-\lim\limits_{\xi\to\infty} {\wh Q}(\xi)=-I$);
\item[2)] $\lim\limits_{\xi\uparrow-1}({\wh Q}(\xi)h,h)=+\infty$ for all $h\in\sN\setminus\{0\}$ (respectively, $s-\lim\limits_{\xi\uparrow -1} {\wh Q}(\xi)=0$);
   \item[3)] $s-\lim\limits_{\xi\downarrow 1} {\wh Q}(\xi)=0$ (respectively, $\lim\limits_{\xi\downarrow1}({\wh Q}(\xi)h,h)=-\infty$
          for all $h\in\sN\setminus\{0\}$).
\end{enumerate}
\end{definition}

The function $Q_\mu$ belongs $\sS_\mu(\sN)$ while $Q_M$ is of the
class $\sS_M(\sN)$. It is stated in \cite{KrO} that if the function
$\wh Q$ belongs to $\sS_\mu(\sN)$  (respectively, $\wh
Q\in\sS_M(\sN)$), then it is a $Q_\mu$-function (respectively,
$Q_M$-function) of some nondensely defined Hermitian contraction
$B$. However, it is shown in \cite{AHS} that this statements is true
only when $\dim\sN<\infty$.

\begin{theorem}
\label{Theor}{\rm{\cite{AHS}}}. Assume that ${\wh Q}\in\sS_\mu(\sN)$
($\wh Q\in \sS_M(\sN)$). Then there exist a Hilbert space $\sH$
containing $\sN$ as a subspace, a Hermitian contraction $B$ in $\sH$
defined on $\dom B=\sH\ominus\sN$, and a pair $\{\wh B_0,\wh B_1\}$
of $sc$-extensions of $B$, satisfying $ \wh B_0\le \wh B_1$, $\ker
(\wh B_1-\wh B_0)=\dom B$, such that $\wh Q(\xi)$ admits the
representation in the form \eqref{q1} (in the form \eqref{q2},
respectively). Moreover, the pair $\{\wh B_0,\wh B_1\}$ possesses
the following properties
\begin{equation}
\label{parsum}
 \ran (\wh B_1-\wh B_0)^{1/2}\cap\ran (\wh B_0-B_\mu)^{1/2}=\ran (\wh B_1-\wh B_0)^{1/2}\cap\ran
(B_M-\wh B_1)^{1/2}=\{0\},
\end{equation}
If $\dim\sN<\infty$, then  necessarily $\wh B_0=B_\mu$ and $\wh
B_1=B_M$.
\end{theorem}

In particular, in the case that $\dim \sN=\infty$ \cite{AHS} (see
also \cite{ArlHassi_2014}) contains a construction of pairs $\{\wh
B_0,\wh B_1\}$ of $sc$-extensions which differ from $\{B_\mu, B_M\}$
and satisfy the conditions in \eqref{parsum}: in other words, the
corresponding $Q$-functions given by \eqref{q1} and \eqref{q2}
belong to $\sS_\mu(\sN)$ and $\sS_M(\sN)$, respectively, but they do
not coincide with the $Q_\mu$- and $Q_M$-functions of $B$.

To finish this section the following simple observation is
mentioned: if $V$ is an isometry in $\sN$ and $\wh B_0\le \wh B_1$
are $sc$-extensions, then the operator-valued functions
\[
\begin{array}{l}
\wh Q_0(\xi):=\left(I_{\sN}+V(\wh B_1-\wh B_0)^{1/2}\left(\wh
 B_0-\xi I_\sH\right)^{-1}(\wh B_1-\wh
 B_0)^{1/2}V^*\right)\uphar\sN,\\
 \wh Q_1(\xi):=\left(-I_{\sN}+V(\wh B_1-\wh B_0)^{1/2}\left(\wh
 B_1-\xi I_\sH\right)^{-1}(\wh B_1-\wh
B_0)^{1/2}V^*\right)\uphar\sN,\; \xi\in \Ext[-1,1]
\end{array}
\]
belong to the Herglotz-Nevanlinna class and $\wh Q^{-1}_1(\xi)=-\wh
Q_0(\xi),$  $\xi\in \Ext[-1,1].$
\begin{remark}
\label{paralslozh} If $F$ and $G$ are bounded nonnegative
selfadjoint operators, then the parallel sum $F:G$ can be
 defined \cite{AD}, \cite{FW}. The conditions
$F:G=0$ and $\ran F^{1/2}\cap\ran G^{1/2}=\{0\}$ are equivalent.
\end{remark}

\section{Selfadjoint contractive extensions of nondensely defined Hermitian contractions with exit}

Let $B$ be a nondensely defined Hermitian contraction in the Hilbert
space $\sH$ and let $\cH$ be an auxiliary Hilbert space. If $B$ is
given by \eqref{HC}, then all $qsc$-extensions of $B$ in the
extended Hilbert space $\sH\oplus\cH$ can be described as follows.
Let
\[
\wh \cH=\sN\oplus\cH,
\]
let $j_{\wh\cH}$ be the canonical embedding operator $\sN\to\wh\cH$,
and define $\wh K_0=j_{\wh\cH}K_0$.
 Then
 \[
D_{\wh K_0^*}=(I_{\wh\cH}-\wh K_0\wh K_0^*)^{1/2}=\begin{bmatrix}D_{
K_0^*}&0\cr 0&I_\cH
\end{bmatrix}:\begin{array}{l}\sN\\\oplus\\\cH \end{array}\to \begin{array}{l}\sN\\\oplus\\\cH \end{array}.
 \]
Clearly, $\sD_{\wh K_0^*}=\sD_{ K_0^*}\oplus\cH\subset \wh\cH$. In
what follows we identify $B$ with its image in $\sH_0\oplus\wh\cH$.
By Theorem \ref{qscex} a $qsc$-extension $\wt B$ of $B$ in
$\sH\oplus\cH$ with respect to the decomposition
$\sH\oplus\cH=\sH_0\oplus\wh \cH$ takes the block form
\[
\wt B=\wt B_X=\begin{bmatrix}B_0&D_{B_0}\wh K_0^*\cr \wh K_0D_{B_0}&
-\wh K_0B_0\wh K_0^*+D_{\wh K_0^*}XD_{\wh
K_0^*}\end{bmatrix}:\begin{array}{l}\sH_0\\\oplus\\\wh \cH
\end{array}\to \begin{array}{l}\sH_0\\\oplus\\\wh\cH \end{array},
 \]
where $X:\sD_{\wh K_0^*}\to \sD_{\wh K_0^*}$ is a contraction. Let
\begin{equation}
\label{blx}
X=\begin{bmatrix}X_{11}& X_{12}\cr X_{21}&X_{22}  \end{bmatrix}:\begin{array}{l}\sD_{ K_0^*}\\
\oplus\\\cH\end{array}\to \begin{array}{l}\sD_{ K_0^*}\\
\oplus\\\cH\end{array}
\end{equation}
be the block representation of the operator $X$. Then
\begin{equation}\label{CONTREXT1}
 \wt B=\begin{bmatrix}B_0&D_{B_0}K_0^*&0\cr K_0
D_{B_0}&-K_0B_0K_0^*+D_{K_0^*}X_{11}D_{K_0^*}& D_{K_0^*}X_{12}\cr
0&X_{21}D_{K_0^*}&X_{22}\end{bmatrix}
:\begin{array}{l}\sH_0\\\oplus\\\sN\\\oplus\\\cH\end{array}\to
\begin{array}{l}\sH_0\\\oplus\\\sN\\\oplus\\\cH\end{array}.
\end{equation}
Let $\cK$ be a Hilbert space. Associate with any selfadjoint
contraction
\[X=\begin{bmatrix}X_{11}& X_{12}\cr X^*_{12}&X_{22}  \end{bmatrix}:\begin{array}{l}\cK\\
\oplus\\\cH\end{array}\to \begin{array}{l}\cK\\
\oplus\\\cH\end{array}
\]
two further selfadjoint contractions in $\cK$ via
\begin{multline}
\label{predeli} \wh Z_0:=\left((I+X)_{\cK}-I\right)\uphar\cK
=X_{11}-\left((I+X_{22})^{(-1/2)}X^*_{12}\right)^*(I+X_{22})^{(-1/2)}X^*_{12},\\
\wh Z_1:=\left(I-(I-X)_{\cK}\right)\uphar\cK
=X_{11}+\left((I-X_{22})^{(-1/2)}X^*_{12}\right)^*(I-X_{22})^{(-1/2)}X^*_{12}.
\end{multline}
By Remark \ref{parr} selfadjoint contractions $X$ in $\cK\oplus\cH$
are of the form
\begin{equation}\label{XX}
X=\begin{bmatrix} -UX_{22}U^*+D_{U^*}VD_{U^*}& UD_{X_{22}}\cr
D_{X_{22}}U^*&X_{22}\end{bmatrix}:\begin{array}{l}\cK\\\oplus\\\cH\end{array}\to
\begin{array}{l}\cK\\\oplus\\\cH\end{array},
\end{equation}
where $X_{22}\in\bL(\cH)$, $U\in\bL(\sD_{X_{22}},\cK)$,
$V\in\bL(\sD_{U^*})$ are contractions, and $X_{22}$ and $V$ are
selfadjoint. Then from \eqref{predeli} and \eqref{XX} one obtains
\[
\wh Z_0=D_{U^*}VD_{U^*}-UU^*,\quad \wh Z_1=D_{U^*}VD_{U^*}+UU^*.
\]
Hence,
\begin{equation}
\label{aft} UU^*=\frac{1}{2}(\wh Z_1-\wh Z_0), \quad
D_{U^*}VD_{U^*}=\frac{1}{2}(\wh Z_1+\wh Z_0).
\end{equation}
Then clearly $\wh Z_0\le \wh Z_1$ and, moreover,
\begin{equation}
\label{eqzer} \ker (\wh Z_1-\wh Z_0)=\{0\}\iff \ker X^*_{12}=\{0\}.
\end{equation}
With $\cK=\sD_{K^*_0}$ as above, $\wh Z_0$ and $\wh Z_1$ determine
two $sc$-extensions $\wh B_0$ and $\wh B_1$ of $B$ in $\sH$:
\begin{equation}
\label{BPHIPM}
 \wh B_0: =\begin{bmatrix}B_0&D_{B_0}K_0^*\cr
K_0 D_{B_0}&-K_0B_0K_0^*+D_{K_0^*}\wh Z_0D_{K_0^*}\end{bmatrix}
 =B_\mu+D_{K^*_0}(I+\wh Z_0)D_{K^*_0},
\end{equation}
\begin{equation}
\label{BPHIPM1} \wh B_1: =\begin{bmatrix}B_0&D_{B_0}K_0^*\cr K_0
D_{B_0}&-K_0B_0K_0^*+D_{K_0^*}\wh Z_1D_{K_0^*}\end{bmatrix}
=B_M-D_{K^*_0}(I-\wh Z_1)D_{K^*_0},
\end{equation}
From definitions and Remark \ref{parr} we get
$$ (I+\wh B_0)_\sN=D_{K^*_0}
P_\sN(I+X)_{\sD_{K^*_0}}D_{K^*_0}P_\sN,\;(I-\wh
B_1)_\sN=D_{K^*_0}P_\sN(I-X)_{\sD_{K^*_0}} D_{K^*_0}P_\sN.$$

\begin{proposition}
\label{parr11} Let $\wh Z_0$ and $\wh Z_1$ be two selfadjoint
contractions in a Hilbert space $\cK$, such that $\wh Z_0\le \wh
Z_1$. If the Hilbert space $\cH$ satisfies $\dim\cH\ge\dim\cran(\wh
Z_1-\wh Z_0)$, then all selfadjoint contractions $X$ in
$\cK\oplus\cH$ possessing the properties
$\left((I+X)_{\cK}-I\right)\uphar\cK=\wh Z_0$ and
$\left(I-(I-X)_{\cK}\right)\uphar\cK=\wh Z_1$ are given by the
formula
\[\begin{array}{l}
X=\begin{bmatrix} \frac{\wh Z_1+\wh Z_0}{2}-\left(\frac{\wh Z_1-\wh
Z_0}{2}\right)^{1/2}\cV^*X_{22}\cV\left(\frac{\wh Z_1-\wh
Z_0}{2}\right)^{1/2}& \left(\frac{\wh Z_1-\wh
Z_0}{2}\right)^{1/2}\cV^*D_{X_{22}} \cr D_{X_{22}}\cV\left(\frac{\wh
Z_1-\wh Z_0}{2}\right)^{1/2}&X_{22}\end{bmatrix}:
\begin{array}{l}\cK\\\oplus\\\cH\end{array}\to
\begin{array}{l}\cK\\\oplus\\\cH\end{array}
\end{array}
\]
where $X_{22}$ is an arbitrary selfadjoint contraction in $\cH$ and
$\cV$ is an arbitrary isometry from $\cran (\wh Z_1-\wh Z_0)$ into
$\sD_{X_{22}}$. In particular, if $\wh Z_0=-I_{\cK}$ and $\wh
Z_1=I_{\cK}$, then
\[\begin{array}{l}
X=\begin{bmatrix} -\cV^*X_{22}\cV& \cV^*D_{X_{22}} \cr
D_{X_{22}}\cV&X_{22}\end{bmatrix}:\begin{array}{l}\cK\\\oplus\\\cH\end{array}\to
\begin{array}{l}\cK\\\oplus\\\cH\end{array},
\end{array}
\]
where $\cV$  is an arbitrary isometry from $\cK$ into
$\sD_{X_{22}}.$
\end{proposition}
\begin{proof}
Conclusions in the proposition follow from relations
\eqref{predeli}, \eqref{XX}, and \eqref{aft}.
\end{proof}

The next result clarifies the definitions of $\wh B_0$ and $\wh B_1$
in \eqref{BPHIPM}, \eqref{BPHIPM1} by establishing an exit space
version for the identities in \eqref{rn1}.

\begin{theorem}
\label{shorts11} Assume that $\sD_{K^*_0}=\sN$, let
$X=(X_{ij})_{i,j=1}^2$ be a selfadjoint contraction in
$\sN\oplus\cH$ as in \eqref{blx}, and let
\[
\wt B_X= \begin{bmatrix}B_0&D_{B_0}\wh K_0^*\cr \wh K_0D_{B_0}& -\wh
K_0B_0\wh K_0^*+D_{\wh K_0^*}XD_{\wh
K_0^*}\end{bmatrix}:\begin{array}{l}\sH\\\oplus\\\cH\end{array}\to\begin{array}{l}\sH\\\oplus\\\cH\end{array}.
\]
Then $\wh B_0$ and $\wh B_1$ defined in \eqref{BPHIPM} and
\eqref{BPHIPM1} satisfy the relations
\begin{equation}\label{gjlheujve}
 \wh B_0=B_\mu+ \left(I+\wt B_X\right)_\sN\uphar\sH,\quad \wh B_1=B_M-\left(I-\wt B_X\right)_\sN\uphar\sH.
\end{equation}
\end{theorem}
\begin{proof}
Let $ \wt B_\mu:=B_\mu P_\sH\oplus (- P_\cH),$ $\wt B_M:=B_M
P_\sH\oplus P_\cH. $ Then it follows from \eqref{ENDS} that
\[
\wt B_X=\wt B_\mu+D_{\wh K^*_0}(I+X)D_{\wh K^*_0}=\wt B_M-D_{\wh
K^*_0}(I-X)D_{\wh K^*_0}.
\]
Moreover, using \eqref{Sh1} and \eqref{extr} it is seen that for all
$f\in\sH\oplus\cH$
\begin{multline*}
 \left(\left(I+\wt B_X\right)_\sN f,f\right)= \inf\limits_{\begin{array}{l}f_0\in\sH_0\\
  h\in\cH\end{array}}\left(\left(I+\wt B_X\right)(f+f_0+h),f+f_0+h\right)\\
 =\inf\limits_{f_0\in\sH_0}\left(\left(I+B_\mu\right)(f+f_0),f+f_0\right)
 +\inf\limits_{h\in\cH}\left(\left(I+X\right)D_{\wh K^*_0}(f+h),D_{\wh K^*_0}(f+h)\right)\\
 =\inf\limits_{h\in\cH}\left(\left(I+X\right)D_{\wh K^*_0}(f+h),D_{\wh K^*_0}(f+h)\right)
 =\left(\left(I+X\right)_{\sN}D_{K^*_0}P_{\sN}f,D_{K^*_0}P_{\sN}f\right).
\end{multline*}
In view of \eqref{predeli} $(I+X)_\sN=I+\wh Z_0$ which combined with
the identity \eqref{BPHIPM} leads to
\[
D_{K^*_0}(I+\wh Z_0)D_{K^*_0}=\wh B_0-B_\mu.
\]
This proves the first identity in \eqref{gjlheujve}. The second
identity in \eqref{gjlheujve} is proved similarly.
\end{proof}

It is also useful to describe shortenings of $I\pm \wt B_X$ to the
exit space $\cH$.

\begin{theorem}\label{equshorts}
Let $X=(X_{ij})_{i,j=1}^2$ be a selfadjoint contraction in
$\sD_{K_0^*}\oplus\cH$ as in \eqref{blx} and let
\[
 \wt B_X=\begin{bmatrix}B_0&D_{B_0}\wh K_0^*\cr \wh K_0 D_{B_0}&-\wh
 K_0B_0\wh K_0^*+D_{\wh K_0^*}X D_{\wh K_0^*}\end{bmatrix}
 :\begin{array}{l}\sH_0\\\oplus\\\wh\cH\end{array}\to
 \begin{array}{l}\sH_0\\\oplus\\\wh\cH\end{array}.
\]
Then
 \begin{equation}
 \label{priyat}
 (I\pm \wt B_X)_\cH\uphar\cH=(I\pm X)_\cH\uphar\cH.
 \end{equation}
 \end{theorem}
 \begin{proof}
Rewrite $\wt B_X$ as in \eqref{CONTREXT1}:
\[
\wt B_X=\begin{bmatrix}B_0&D_{B_0}K_0^*&0\cr K_0
D_{B_0}&-K_0B_0K_0^*+D_{K_0^*}X_{11}D_{K_0^*}& D_{K_0^*}X_{12}\cr
0&X^*_{12}D_{K_0^*}&X_{22}\end{bmatrix}
:\begin{array}{l}\sH_0\\\oplus\\\sN\\\oplus\\\cH\end{array}\to
\begin{array}{l}\sH_0\\\oplus\\\sN\\\oplus\\\cH\end{array}.
\]
Let $\cX$ be the Hermitian contraction determined by the first
column of $X$,
\[
\cX=\begin{bmatrix}X_{11}\cr X_{12}^*
\end{bmatrix}:\sD_{ K_0^*}\to\begin{array}{l}\sD_{ K_0^*}\\\oplus\\\cH\end{array}.
\]
Then one can consider $X$ as an $sc$-extension of $\cX$.
Analogously, define the Hermitian contraction $\cB_\cX$ by
\[
\cB_\cX=\begin{bmatrix}B_0&D_{B_0}K_0^*\cr K_0
D_{B_0}&-K_0B_0K_0^*+D_{K_0^*}X_{11}D_{K_0^*}\cr
0&X_{12}^*D_{K_0^*}\end{bmatrix}
:\begin{array}{l}\sH_0\\\oplus\\\sN\end{array}\to
\begin{array}{l}\sH_0\\\oplus\\\sN\\\oplus\\\cH\end{array}.
\]
Now we consider $sc$-extensions of $\cX$ in the Hilbert space
$\wh\cH=\sD_{K_0^*}\oplus\cH$ and $sc$-extensions of $\cB_\cX$ in
the Hilbert space $\sH\oplus\cH=\sH_0\oplus\sN\oplus\cH.$ It is
evident that
\[
\wt B_X\supset\cB_\cX\iff X\supset\cX.
\]
All $sc$-extensions of $\cX$ form the operator interval
$[(\cX)_\mu,(\cX)_M]$. On the other hand, the form of $\wt B_X$
shows that
\[
 X_1\leq X_2 \iff \wt B_{X_1} \leq \wt B_{X_2}.
\]
Hence,
\[
 X\in [(\cX)_\mu,(\cX)_M] \,\Rightarrow\, \wt B_X\in[\wt B_{(\cX)_\mu},\wt B_{(\cX)_M}].
\]

On the other hand, every $sc$-extension $\wt B$ of $B$ in
$\sH\oplus\cH$ is of the form $\wt B_X$, where $X$ is a selfadjoint
contraction in $\sD_{K_0^*}\oplus\cH$; see \eqref{blx},
\eqref{CONTREXT1}. It follows that if $\wt B$ is an $sc$-extension
of $B_\cX$, then $\wt B$ is also an $sc$-extension of $B\,(\subset
B_\cX)$, i.e., $\wt B=\wt B_X,$  where $X$ is the $sc$-extension of
$\cX.$ Hence,
\[
\wt B_X\in [(\cB_\cX)_\mu, (\cB_\cX)_M] \,\Rightarrow\,
X\in[(\cX)_\mu,(\cX)_M].
\]
One concludes that
\begin{equation}\label{extBX}
 (\cB_\cX)_\mu=\wt B_{(\cX)_\mu},\quad (\cB_\cX)_M=\wt B_{(\cX)_M}.
\end{equation}
Since for all $X_1, X_2\in[(\cX)_\mu,(\cX)_M]$ one has
\[
 \left(\wt B_{X_1}-\wt B_{X_2}\right)\uphar\cH=\left(X_1-X_2\right)\uphar\cH,
\]
the equalities \eqref{rn1} applied to $I\pm \wt B_X$ and $(I\pm X)$
yield \eqref{priyat} in view of \eqref{extBX}.
\end{proof}
\begin{corollary}\label{ochvazh}
The following statements are equivalent:
\begin{enumerate}
\def\labelenumi{\rm (\roman{enumi})}
\item $(I+ \wt B_X)_\cH=0$ and $(I- \wt B_X)_\cH=0$;
\item $(I+ X)_\cH=0$ and $(I- X)_\cH=0$;
\item $\wt B_X$ is a unique sc-extension of Hermitian contraction  $\cB_\cX$;
\item $X$ is a unique sc-extension of Hermitian contraction  $\cX$.
\end{enumerate}
\end{corollary}

Theorem~\ref{equshorts} and Corollary~\ref{ochvazh} have important
implications on the contractions $\wh Z_0$, $\wh Z_1$, therefore,
also on the $sc$-extensions $\wh B_0$, $\wh B_1$ of $B$ in the
original Hilbert space $\sH$.

\begin{theorem}
\label{novaya} Let
\[
 X=\begin{bmatrix}X_{11}& X_{12}\cr X^*_{12}&X_{22}
\end{bmatrix}:\begin{array}{l}\\\cK\\\oplus\\\cH\end{array}\to\begin{array}{l}\\\cK\\\oplus\\\cH\end{array}
\]
be a selfadjoint contraction. Suppose that
\begin{equation}
\label{uslovvazh1} (I-X)_\cH=(I+X)_\cH=0,
\end{equation}
and
\begin{equation}
\label{esche} ||X_{22}||<1.
\end{equation}
Let the selfadjoint contractions $\wh Z_0$ and $\wh Z_1$ in $\cK$ be
defined by \eqref{predeli}. Then
\begin{equation}
\label{novrabstva} \begin{array}{l} \ran(\wh Z_1-\wh
Z_0)^{1/2}\cap\ran(I+\wh Z_0)^{1/2}=\{0\}, \\
\ran(\wh Z_1-\wh Z_0)^{1/2}\cap\ran(I-\wh Z_1)^{1/2}=\{0\}.
\end{array}
\end{equation}
\end{theorem}
\begin{proof} By \eqref{predeli} we have
\[
 I_{\cK}+\wh Z_0=(I+X)_{\cK}\uphar\cK,\quad
 I_{\cK}-\wh Z_1=(I-X)_{\cK}\uphar\cK.
\]
Due to the assumption \eqref{uslovvazh1}, the operator $X$ takes the
form
\[
X=\begin{bmatrix} X_{11}& D_{X_{11}}L^*\cr LD_{X_{11}}&
-LX_{11}L^*\end{bmatrix}
:\begin{array}{l}\cK\\\oplus\\\cH\end{array}\to
\begin{array}{l}\cK\\\oplus\\\cH\end{array},
\]
where $LL^*=I_\cH$; see Remark~\ref{parr}. On the other hand,
$X_{12}=D_{X_{11}}L^*= UD_{X_{22}}$ for a contraction
$U\in\bL(\sD_{X_{22}},\cK)$. From the assumption \eqref{esche} it
follows that $D_{X_{22}}$ has a bounded inverse. Hence
$U=D_{X_{11}}L^*D^{-1}_{X_{22}}$ and
\[
\ran U=D_{X_{11}}\ran L^*.
\]
Furthermore, since $\wh Z_1-\wh Z_0=2 UU^*$, see \eqref{aft}, one
obtains
\[
\ran(\wh Z_1-\wh Z_0)^{1/2}=\ran U=D_{X_{11}}\ran L^*.
\]

On the other hand, from the formula for $X$ above it is clear that
\[
 I\pm X=\begin{bmatrix} (I\pm X_{11})^{1/2} \\ L(I\mp X_{11})^{1/2}\end{bmatrix}
  \begin{bmatrix} (I\pm X_{11})^{1/2} \\ L(I\mp X_{11})^{1/2}\end{bmatrix}^*.
\]
This gives a description of $\ran (I\pm X)^{1/2}$ and now an
application of \eqref{Sh2} leads to
\[
 \ran(I+\wh Z_0)^{1/2}=(I+X_{11})^{1/2}(I- X_{11})^{-1/2}\ker L,
\]
\[
\ran(I-\wh Z_1)^{1/2}=(I-X_{11})^{1/2}(I+ X_{11})^{-1/2}\ker L.
\]
Since $\ran L^*\perp\ker L$, one concludes that
\[
\begin{array}{l}
(I-X_{11})^{1/2}\ran L^*\cap(I- X_{11})^{-1/2}\ker L=\{0\},\\
(I+X_{11})^{1/2}\ran L^*\cap(I+ X_{11})^{-1/2}\ker L=\{0\}.
\end{array}
\]
This implies the equalities \eqref{novrabstva}.
\end{proof}

Observe that if $B$ is a Hermitian contraction in $\sH$, if $\wh
Z_0$ and $\wh Z_1$ are selfadjoint contractions in
$\sN(=\sH\ominus\dom B)$ satisfying \eqref{novrabstva}, and if the
$sc$-extensions $\wh B_0$ and $\wh B_1$ of $B$ are given by
\[
\wh B_j=\begin{bmatrix}B_0&D_{B_0}K_0^*\cr K_0D_{B_0}& -K_0B_0
K_0^*+D_{ K_0^*}\wh Z_j D_{ K_0^*} \end{bmatrix},\; j=0,1,
\]
then the pair $\{\wh B_0, \wh B_1\}$ possesses the properties in
\eqref{parsum}. If $\sD_{K^*_0}=\sN$ and $\ker(\wh Z_1-\wh
Z_0)=\{0\}$, then $\ker (\wh B_1-\wh B_0)=\dom B.$ We also note that
if
\[
X=\begin{bmatrix} 0& \cV^* \cr
\cV&0\end{bmatrix}:\begin{array}{l}\cK\\\oplus\\\cH\end{array}\to
\begin{array}{l}\cK\\\oplus\\\cH\end{array},
\]
where $\cV$  is an isometry from $\cK$ into $\cH$, then $(I\pm
X)_\cK=0$ and $\wh Z_0=-I_\cK$, $\wh Z_1=I_\cK$. On the other hand,
$(I\pm X)_\cH=I-\cV\cV^*$ and hence $(I\pm X)_\cH=0$ if and only if
$\cV$ is unitary, i.e., $\ran \cV=\cH$. Therefore, it is possible
that \eqref{esche} and \eqref{novrabstva} are satisfied, while
\eqref{uslovvazh1} fails to hold.

The next result completes the role of exit space extensions in the
study of pairs $\{\wh B_0,\wh B_1\}$ of $sc$-extensions of $B$ in
the original Hilbert space $\sH$ whose $Q$-functions belong to the
classes $\sS_\mu(\sN)$ and $\sS_M(\sN)$; see Definition~\ref{def1}
and Theorem~\ref{Theor}.

\begin{theorem}
\label{construc} 1) Let $\dim\cK=\dim\cH=\infty$. Then there exists
a selfadjoint contractive block operator
\[
 X=\begin{bmatrix}X_{11}& X_{12}\cr X^*_{12}&X_{22}
\end{bmatrix}:\begin{array}{l}\\\cK\\\oplus\\\cH\end{array}\to\begin{array}{l}\\\cK\\\oplus\\\cH\end{array}
\]
satisfying the conditions \eqref{uslovvazh1}, \eqref{esche}, and the
additional conditions
\begin{equation}
\label{nullzer} \ker X^*_{12}=\{0\},
\end{equation}
$$ \wh Z_0\ne -I_\cK,\; \wh Z_1\ne I_\cK,\;
\ker (\wh Z_1-\wh Z_0)=\{0\},$$ where $\wh Z_0$ and $\wh Z_1$ are as
in \eqref{predeli}, i.e., $\wh
Z_0=\left((I+X)_{\cK}-I\right)\uphar\cK$, $\wh
Z_1=\left(I-(I-X)_{\cK}\right)\uphar\cK.$

2) Let $\dim\cK=\infty$ and suppose that $\wh Z_0$ and $\wh Z_1$,
$\wh Z_0\leq \wh Z_1$, are two selfadjoint contractions in $\cK$
which satisfy the conditions \eqref{novrabstva} and the condition
\begin{equation}
\label{zk} \ker(\wh Z_1-\wh Z_0)=\{0\}.
\end{equation}
Then there exists a selfadjoint contractive block operator $X$ in
the Hilbert space $\cK\oplus \cH$, $\dim \cH=\dim \cK$, such that
\[
 (I\pm X)_\cH=0, \quad ||X_{22}||<1,
\]
and $\wh Z_0=\left((I+X)_{\cK}-I\right)\uphar\cK$, $\wh
Z_1=\left(I-(I-X)_{\cK}\right)\uphar\cK$.
\end{theorem}
\begin{proof}
1) We give a construction of a required $X$ in two steps.

\textbf{Step 1.} In $\cK$ choose an infinite dimensional subspace
$\Omega_0$ with an infinite dimensional orthogonal complement
$\sM_0=\cK\ominus\Omega_0$. In this step we construct a special
selfadjoint contraction $X_{11}$ in $\cK=\Omega_0\oplus \sM_0$.

Let $\cA$ be a selfadjoint operator in $\Omega_0$ such that
$||\cA||<1$. Then choose a contraction $\cM\in\bL(\Omega_0,\sM_0)$
such that $\ker D_{\cM^*}=\{0\}$ and $\ran D_{\cM^*}\ne \sM_0$,
i.e., $||\cM f||<||f||$ ($\Leftrightarrow$ $\ker D_{\cM}=\{0\}$) for
all $f\in\Omega_0\setminus\{0\}$, while $||\cM||=1$; cf.
\eqref{DTran2}. Moreover, let $\sL_0$ be a subspace in $\sM_0$ such
that
\begin{equation}
\label{inersect11}
 \sL_0\cap\ran D_{\cM^*}=\{0\} \text{ and } \sL_0^\perp \cap\ran D_{\cM^*}=\{0\};
\end{equation}
cf. \cite{schmud}. Next define the selfadjoint and unitary operator
$J_0$ in $\sM_0$ by
\begin{equation}
\label{010}
 J_0=2P_{\sL_0}-I_{\sM_0}.
\end{equation}
Due to \eqref{inersect11} $J_0$ satisfies
\begin{equation}
\label{inersect21} J_0\ran D_{\cM^*}\cap\ran D_{\cM^*}=\{0\}.
\end{equation}
Now, introduce
\[
X_{11}=\begin{bmatrix}\cA& D_{\cA}\cM^*\cr\cM D_{\cA}&
-\cM\cA\cM^*+D_{\cM^*}J_0D_{\cM^*}\end{bmatrix}:\begin{array}{l}\Omega_0\\\oplus\\\sM_0\end{array}\to
\begin{array}{l}\Omega_0\\\oplus\\\sM_0\end{array}.
\]
We claim that $X_{11}$ satisfies the equalities
\begin{equation}
\label{kerzero1} \ker D_{X_{11}}=\{0\}
\end{equation}
and
\begin{equation}
\label{ranzero1} \ran D_{X_{11}}\cap\sM_0=\{0\}.
\end{equation}
Since $J_0$ in \eqref{010} is unitary, $D_{J_0}=0$ and hence Remark
\ref{parr} shows that for all $\vec a=\begin{bmatrix}h\cr
g\end{bmatrix}$
\begin{equation}\label{DX11}
 \left\|D_{X_{11}}\vec a\right\|^2
 =\left\|D_{\cM}\left(D_{\cA} h-\cA\cM^* g\right)-\cM^*J_0D_{\cM^*}g\right\|^2.
\end{equation}
Hence, if $ \left\|D_{X_{11}}\vec a\right\|^2=0$ then it follows
from \eqref{DTran1} that there exists $x\in\sM_0$ such that $D_{\cA}
h-\cA\cM^* g=\cM^* x$ and $J_0D_{\cM^*}g-D_{\cM^*}x\in \ker \cM^*
\subset \ran D_{\cM^*}$. Now \eqref{inersect21} gives
$J_0D_{\cM^*}g=0$ and, hence, $g=0$ and $h=0$, since also $\ker
D_A=0$. So \eqref{kerzero1} holds true.

On the other hand, by applying \eqref{Sh1} to \eqref{DX11} it is
seen that $\left(D^2_{X_{11}}\right)_{\sM_0}=0$, and hence
\eqref{ranzero1} is obtained from \eqref{nol}. Furthermore,  an
application of Remark \ref{parr} shows that when $\sL_0\neq
\{0\}\neq \sL_0^\perp$ then, equivalently,
\begin{equation}
 \label{nerav}
 (I+X_{11})_{\sM_0}\ne 0,\; (I-X_{11})_{\sM_0}\ne 0.
\end{equation}
\textbf{Step 2.} Let $\cH$ be an infinite-dimensional Hilbert space,
let $L^*$ be an isometry from $\cH$ into $\cK$ such that $\ran
L^*=\Omega_0$, and define
\[
X=\begin{bmatrix}X_{11}& D_{X_{11}}L^*\cr LD_{X_{11}}&-LX_{11}L^*
\end{bmatrix}:\begin{array}{l}\cK\\\oplus\\\cH\end{array}\to
\begin{array}{l}\cK\\\oplus\\\cH\end{array}.
\]
It follows from $\ran L^*=\Omega_0$ that $X_{22}=-LP_{\Omega_0}\cA
L^*$ and thus $||X_{22}||<1$ by the choice of $\cA$. Since $\ker L
=\sM_0$, the equalities \eqref{kerzero1}, \eqref{ranzero1} yield
$\ker LD_{X_{11}}=\{0\}$. Therefore,
\[
(I\pm X)_\cH=0,\; ||X_{22}||<1,\; \ker X^*_{12}=\{0\}, \; \ker
X_{12}=\{0\},
\]
where the first equality holds by Remark \ref{parr}. By applying
Theorem \ref{novaya} one concludes that $\wh Z_0$ and $\wh Z_1$ have
properties \eqref{novrabstva}. Moreover, from \eqref{nerav} it
follows that $ \wh Z_0\ne -I_{\cK},$ $\wh Z_1\ne I_{\cK}$. Thus,
relations in \eqref{uslovvazh1}, \eqref{esche}, \eqref{nullzer} are
valid for $X$. Finally, the condition $\ker(\wh Z_1-\wh Z_0)=\{0\}$
is obtained from \eqref{nullzer} and \eqref{eqzer}.

2) Define $X$ by
\[
X=\begin{bmatrix} \frac{\wh Z_1+\wh Z_0}{2}& \left(\frac{\wh Z_1-\wh
Z_0}{2}\right)^{1/2}\cV^* \cr \cV\left(\frac{\wh Z_1-\wh
Z_0}{2}\right)^{1/2}&0\end{bmatrix}:
\begin{array}{l}\cK\\\oplus\\\cH\end{array}\to
\begin{array}{l}\cK\\\oplus\\\cH\end{array},
\]
where $\cV:\cH\to\cK$ is unitary. Clearly, $X$ is a selfadjoint
contraction in $\cH\oplus\cK$; cf. Theorem \ref{qscex}, Proposition
\ref{parr11}. Next observe that
\[
 I+X=\begin{bmatrix} (I+\wt Z_0)^{1/2} & \left(\frac{\wh Z_1-\wh Z_0}{2}\right)^{1/2}\cV^*\cr 0 & I\end{bmatrix}
     \begin{bmatrix} (I+\wt Z_0)^{1/2} & \left(\frac{\wh Z_1-\wh Z_0}{2}\right)^{1/2}\cV^*\cr 0 & I\end{bmatrix}^*
\]
and
\[
 I-X=\begin{bmatrix} (I-\wt Z_0)^{1/2} & -\left(\frac{\wh Z_1-\wh Z_0}{2}\right)^{1/2}\cV^*\cr 0 & I\end{bmatrix}
     \begin{bmatrix} (I-\wt Z_0)^{1/2} & -\left(\frac{\wh Z_1-\wh Z_0}{2}\right)^{1/2}\cV^*\cr 0 & I\end{bmatrix}^*.
\]
These two formulas give descriptions for $\ran(I+X)^{1/2}$ and
$\ran(I-X)^{1/2}$, respectively. Now using the assumptions
\eqref{novrabstva} and \eqref{zk} one concludes that
\[
 \ran(I\pm X)^{1/2}\cap (\{0\}\oplus \cH)=\{0\}.
\]
According to \eqref{Sh2} this means that $(I\pm X)_\cH=0$.

Finally, the equalities $\wh
Z_0=\left((I+X)_{\cK}-I\right)\uphar\cK$ and $\wh
Z_1=\left(I-(I-X)_{\cK}\right)\uphar\cK$ are clear from Proposition
\ref{parr11}.
\end{proof}

In particular, Theorem \ref{construc} contains an improvement of
Theorem~\ref{Theor}: given any Hermitian contraction $B$ in $\sH$
with $\dim \sN=\dim(\sH\ominus\dom B)=\infty$ it enables to
construct pairs $\{\wh B_0, \wh B_1\}$ of $sc$-extensions of $B$ in
$\sH$, which differ from the pair $\{\wh B_\mu, \wh B_M\}$ and
satisfy the conditions \eqref{parsum}, directly from one exit space
extension $\wt B_X$ of $B$ via the formulas
\eqref{CONTREXT1}--\eqref{BPHIPM1}. Furthermore, all the key
properties of $\{\wh B_0, \wh B_1\}$ are expressed in simple terms
and the choice of appropriate parameters $X$ is specified
explicitly.

\section{Compressed resolvents}

Let $B$ be Hermitian contraction in $\sH$ and let $\wt B$ be a
$qsc$-extension of $B$ in the Hilbert space $\sH\oplus\cH$. Recall
that then $\wt B=\wt B_X$ can be rewritten in the form
\eqref{CONTREXT1} for some contractive block operator $X$ of the
form \eqref{blx}. To formulate the next result it is useful to
associate with $X$ the operator function
\begin{equation}\label{FX}
 \Phi_X(z)=X_{11}+zX_{12}(I-zX_{22})^{-1}X_{21},\quad |z|<1.
\end{equation}
If, in addition, $X_{22}$ is selfadjoint, then $\Phi_X(z)$ admits a
holomorphic continuation to all points
$z\in\dC\setminus\{(-\infty,-1]\cup[1,+\infty)\}$. Observe, that
$\Phi_X(z)$ can be also interpreted as the transfer function of the
passive system
\[
\sigma=\left\{\begin{bmatrix} X_{11}&X_{12} \cr
X_{21}&X_{22}\end{bmatrix};\sD_{K^*_0},\sD_{K^*_0}, \cH\right\},
\]
see Section \ref{secS}. In particular, $\Phi_X(z)$ is contractive on
the unit disk $\dD$.

\begin{theorem}
\label{comrescontr} Let $B$ be Hermitian contraction in $\sH$, let
$\wt B=\wt B_X$ be a $qsc$-extension of $B$ in $\sH\oplus\cH$
rewritten in the form \eqref{CONTREXT1} with $X$ given by
\eqref{blx}, and let $\Phi_X(z)$ be as in \eqref{FX}. Then
\begin{equation}
\label{genres}
 P_\sH(z\wt B-I)^{-1}\uphar\sH =\left(z\wh B_X(z)-I_\sH\right)^{-1},\quad
 |z|<1,
\end{equation}
where
\begin{equation}\label{BPHI}
\begin{array}{l}
 \wh B_X(z)=\dfrac{1}{2}\,(B_\mu+B_M)+\dfrac{1}{2}\,(B_M-B_\mu)^{1/2}\Phi_X(z)(B_M-B_\mu)^{1/2}\\[3mm]
 \qquad\quad=\begin{bmatrix}B_0&D_{B_0}K_0^*\cr K_0 D_{B_0}&-K_0B_0K_0^*+D_{K_0^*}\Phi_X(z)D_{K_0^*}\end{bmatrix}.
\end{array}
\end{equation}
With $z$ fixed, the operator $\wh B_X(z)$ is a $qsc$-extension of
$B$ in the Hilbert space $\sH$.

Furthermore, if $\wt B\in C_{\sH\oplus\cH}(\alpha)$, then
$\Phi_X(z)$ and $\wh B_X(z)$ can be defined for $z\in \Pi(\alpha)$
and
\begin{enumerate}
\item the implications
\[
\begin{array}{l}
\left\{\begin{array}{l}z\in\Pi_+(\beta),\; z\ne\pm 1,\\\beta\in
[\alpha,\pi/2)
\end{array}\right. \Longrightarrow ||\wh B_X(z)\sin\beta+i\cos\beta||\le
1,\\
\left\{\begin{array}{l}z\in\Pi_-(\beta),\; z\ne\pm 1,\\\beta\in
[\alpha,\pi/2)
\end{array}\right. \Longrightarrow ||\wh B_X(z)\sin\beta-i\cos\beta||\le
1,
\end{array}
\]
are valid, therefore
\[
z\in C(\beta)\Longrightarrow B_X(z)\in C_\sH(\beta);
\]
\item there exist strong limits
$$\Phi_X(\pm
1)\in C_{\sD_{K^*_0}}(\alpha),\quad \wh B_X(\pm 1)\in
C_{\sH}(\alpha).$$
\end{enumerate}

In particular, for $\alpha=0$ the operator functions $\Phi_X(z)$ and
$\wh B_X(z)$ are defined for
$z\in\dC\setminus\{(-\infty,-1]\cup[1,+\infty)\}$ and $\Phi_X(\pm
1)$ are selfadjoint contractions given by
\begin{equation}
\label{predeli11} \Phi_X(-1)=\wh Z_0,\quad \Phi_X(1)=\wh Z_1,
\end{equation}
where $\wh Z_0$ and $\wh Z_1$ are as in \eqref{predeli}, and the
$sc$-extensions $\wh B_X(-1)$ and $\wh B_X(+1)$ of $B$ in $\sH$
coincide with $\wh B_0$ and $\wh B_1$ in \eqref{BPHIPM} and
\eqref{BPHIPM1}, respectively.
\end{theorem}
\begin{proof}
Since $||\Phi(z)||\le 1$, the operator $\wh B_X(z)$ in \eqref{BPHI}
is a $qsc$-extension of $B$ for each $z$, $|z|<1$; see
Theorem~\ref{qscex}. Using \eqref{inv11} and \eqref{CONTREXT1} we
get for $|\lambda|>1$
\begin{multline*}
 \left(P_\sH(\wt B-\lambda)^{-1}\uphar\sH\right)^{-1}
  =\begin{bmatrix}B_0&D_{B_0}K_0^*\cr K_0 D_{B_0}&-K_0B_0K_0^*+D_{K_0^*}X_{11}D_{K_0^*}\end{bmatrix}-\lambda I_\sH\\
 \qquad -\begin{bmatrix}0\cr D_{K^*_0}X_{12}\end{bmatrix}(X_{22}-\lambda)^{-1}
\begin{bmatrix}0&X_{21}D_{K^*_0}\end{bmatrix}\\
=\begin{bmatrix}B_0&D_{B_0}K_0^*\cr K_0
D_{B_0}&-K_0B_0K_0^*+D_{K_0^*}\left(X_{11}-X_{12}(X_{22}-\lambda
)^{-1}X_{21}\right)D_{K_0^*}\end{bmatrix}-\lambda I_\sH.
\end{multline*}
Consequently, with $|z|<1$ this leads to
\[
\begin{array}{ll}
 \left(P_\sH(z\wt B-I)^{-1}\uphar\sH\right)^{-1}
 &=\; z\begin{bmatrix}B_0&D_{B_0}K_0^*\cr K_0
 D_{B_0}&-K_0B_0K_0^*+D_{K_0^*}\Phi_X(z)D_{K_0^*}\end{bmatrix}-
 I_\sH\\[3mm]
 &=\; z\wh B_X(z)-I
\end{array}
\]
and this proves \eqref{genres}.

Suppose that $\wt B\in C_{\sH\oplus\cH}(\alpha)$. Then $X\in
C_{\sD_{K^*_0}\oplus\cH}(\alpha)$ by Theorem \ref{qscex} and this
implies that $X\in C_{\sD_{K^*_0}\oplus\cH}(\beta)$ for
$\beta\in[\alpha,\pi/2)$. Theorem \ref{calsys} combined with
\eqref{XtoB} shows that
\[
\begin{array}{l}
 \left\{\begin{array}{l}z\in\Pi_+(\beta),\\ z\ne\pm 1,\\ \beta\in[\alpha,\pi/2)
 \end{array}\right.
  \Longrightarrow ||\Phi_X(z)\sin\beta+i\cos\beta||\le 1
   \Longrightarrow ||\wh B_X(z)\sin\beta+i\cos\beta||\le 1,\\[7mm]
 \left\{\begin{array}{l}z\in\Pi_-(\beta),\\ z\ne\pm 1,\\\beta\in[\alpha,\pi/2)
 \end{array}\right.
 \Longrightarrow ||\Phi_X(z)\sin\beta-i\cos\beta||\le 1
  \Longrightarrow ||\wh B_X(z)\sin\beta-i\cos\beta||\le 1.
\end{array}
\]
Moreover, according to Theorem \ref{calsys} the strong limit values
$\Phi_X(\pm1)$ exist and in view of \eqref{BPHI} $\wh B_X(\pm 1)$
exist, too, and they satisfy the inclusions in (2). For $\alpha=0$
the equalities in \eqref{predeli11} can be obtained directly from
the formulas in \eqref{predeli} and \eqref{FX}. Finally, by
comparing \eqref{BPHIPM}, \eqref{BPHIPM1} and \eqref{BPHI} one
concludes that $\wh B_X(-1)=\wh B_0$ and $\wh B_X(+1)=\wh B_1$.
\end{proof}

Let $X$ and $\wt B_X$ be given by \eqref{blx} and \eqref{CONTREXT1}.
Define the operator $\wt C$ in $\sH=\sH_0\oplus\sN$ by
 \begin{equation}
 \label{opc}
\wt C=\begin{bmatrix}B_0&D_{B_0}K^*_0 \cr
K_0D_{B_0}&-K_0B_0K_0^*+D_{K_0^*}X_{11}D_{K_0^*}\end{bmatrix}:\begin{array}{l}\sH_0\\\oplus\\\sN\end{array}\to
\begin{array}{l}\sH_0\\\oplus\\\sN\end{array},
\end{equation}
 and let $\wt M:\cH\to\sH$ and its adjoint
$\wt M^*:\sH\to\cH$ be given by
\[
 \wt M=\begin{bmatrix}0\cr D_{K^*_0}X_{12}\end{bmatrix}:\cH\to
 \begin{array}{l}\sH_0\\\oplus\\\sN\end{array},
\quad
 \wt M^*=\begin{bmatrix}0&
 X^*_{12}D_{K^*_0}\end{bmatrix}:\begin{array}{l}\sH_0\\\oplus\\\sN\end{array}\to\cH.
\]
Let the operator $\wt B=\wt B_X$ be given by \eqref{CONTREXT1}. We
rewrite it in the form
\[
\wt B_X=\begin{bmatrix}\wt C&\wt M \cr \wt
M^*&X_{22}\end{bmatrix}:\begin{array}{l}\sH\\\oplus\\\cH\end{array}\to
\begin{array}{l}\sH\\\oplus\\\cH\end{array}.
\]
Consider a passive $pqs$-system $\cT_X=\left\{\wt
B_X;\sH,\sH,\cH\right\} $ with the state space $\cH$ and the
input-output space $\sH$; see Subsection \ref{secS}. The transfer
function of $\cT_X$ is given by
\[
\begin{array}{l}
 \wt C+z\wt M(I-zX_{22})^{-1}\wt M^*\\
 \;=\begin{bmatrix}B_0&D_{B_0}K^*_0 \cr
 K_0D_{B_0}&-K_0B_0K_0^*+D_{K_0^*}X_{11}D_{K_0^*}\end{bmatrix}
 +z\begin{bmatrix}0\cr D_{K^*_0}X_{12}\end{bmatrix}(I-zX_{22})^{-1}
 \begin{bmatrix}0&X^*_{12}D_{K^*_0}\end{bmatrix}\\[3mm]
 \;=\begin{bmatrix}B_0&D_{B_0}K_0^*\cr K_0
 D_{B_0}&-K_0B_0K_0^*+D_{K_0^*}\Phi_X(z)D_{K_0^*}\end{bmatrix}.
\end{array}
\]
Comparing this with \eqref{BPHI} it is seen that the transfer
function of $\cT_X$ is in fact $\wh B_X(z)$.

 Now consider the passive selfadjoint discrete-time system
$ \Sigma_X=\left\{\wt B_X;\cH,\cH,\sH\right\} $ with the state space
$\sH$ and the input-output space $\cH.$ The transfer function
$\Theta$ of the system $\Sigma_X$ is given by
\begin{equation}
\label{transfunc} 
 \Theta(z)=X_{22}+z\wt M^*(I_\sH-z
\wt C)^{-1}\wt M
 =X_{22}+z X^*_{12}D_{K^*_0}P_{\sN}(I_\sH-z \wt
C)^{-1}D_{K^*_0}X_{12}.
\end{equation}
The function
$
Q_{\wt C}(\lambda)=P_{\sN}(\wt C-\lambda I_\sH)^{-1}\uphar\sN,\;
\lambda\in\rho(\wt C) $ is called the $Q$-function \cite{AHS2} of
$\wt C$. Hence,
\[
\Theta(z)=X_{22}- X^*_{12}D_{K^*_0}Q_{\wt
C}(1/z)D_{K^*_0}X_{12},\;1/z\in\rho(\wt C).
\]
The transfer function $\Theta$ possesses the following properties
(see Subsection \ref{secS}):
\begin{enumerate}
\item $\Theta$ belongs to Herglotz-Nevanlinna class for
$z\in \dC\setminus\left((-\infty,-1]\cup [1,+\infty)\right)$,
\item $\Theta$ is a contraction for $z\in
\dD=\{z\in\dC:|z|<1\}$,
\item $\Theta$ has
strong limit values $\Theta(\pm 1)$,
\item if $\beta\in[0,\pi/2)$, then
\begin{equation}\label{propert1}
\begin{array}{l}\left\{\begin{array}{l}
|z\sin\beta+i\cos\beta|\le 1\\
\quad z\ne\pm 1
\end{array}\right.
\Longrightarrow\left\|\Theta(z)\sin\beta+i\cos\beta\,I_\cH\right\|_\cH\le
1,\\
\left\{\begin{array}{l} |z\sin\beta-i\cos\beta|\le 1\\
\quad z\ne\pm 1
\end{array}\right.
\Longrightarrow\left\|\Theta(z)\sin\beta-i\cos\beta\,I_\cH\right\|_\cH\le
1.
\end{array}
\end{equation}
\end{enumerate}
Furthermore, it follows from Theorem \ref{equshorts} and the formula
\eqref{transfunc} that
\[
\begin{array}{l}
\Theta(-1)=(I+\wt
B_X)_\cH\uphar\cH-I_\cH=(I+X)_\cH\uphar\cH-I_\cH,\\
\Theta(1)=I_\cH-(I-\wt B_X)_\cH\uphar\cH=I_\cH-(I-X)_\cH\uphar\cH.
\end{array}
\]
Using the Schur-Frobenius formula \eqref{Sh-Fr1} one gets the
following analog of Theorem \ref{comrescontr}.

\begin{corollary}
\label{compscreas} The relation
\[
P_\cH(z\wt B_X-I)^{-1}\uphar\cH=(z\Theta(z)-I)^{-1},\;
z\in\dC\setminus\{(-\infty,-1]\cup[1,+\infty)\}
\]
is valid.
\end{corollary}
In the next theorem we show that a simple Hermitian contraction $B$
and its $sc$-extension $\wt B$ can be recovered up to the unitary
equivalence by means of $\wh B(z)$ or $\Theta(z)$.

\begin{theorem}
\label{realiz1} 1) Let $\sH$ be a Hilbert space and let the
Herglotz-Nevanlinna function $\wh B(z)$ be from the class ${\bf
S}^{s}(\sH)$. Then there exist a Hermitian contraction $B$ in $\sH$
and its $sc$-extension $\wt B$ in the Hilbert space $\sH\oplus\cH$
such that
$$P_\sH(z\wt B-I)^{-1}\uphar \sH=(z\wh B(z)-I)^{-1}.$$
2) Let $\cH$ be a Hilbert space and let the Herglotz-Nevanlinna
function $\Theta(z)$ be from the class ${\bf S}^{s}(\cH)$. Then
there exist a Hilbert space $\sH$, a simple Hermitian contraction
$B$ in $\sH$ and its $sc$-extension $\wt B$ in the Hilbert space
$\sH\oplus\cH$ such that
$$P_\cH(z\wt B-I)^{-1}\uphar \cH=(z\Theta(z)-I)^{-1}.$$
\end{theorem}
\begin{proof}
1) It is well known that the function $\wh B(z)$ can be realized as
the transfer function of a minimal passive selfadjoint system
\[
\cT=\left\{\begin{bmatrix}\wt C& \wt M\cr \wt M^*&  Y
\end{bmatrix},\; \sH,\sH,\cH\right\}
\]
with input-output space $\sH$ and the state space $\cH$. Here the
operator
\[
\wt B= \begin{bmatrix}\wt C& \wt M\cr \wt M^*&  Y
\end{bmatrix}:\begin{array}{l}\sH\\\oplus\\\cH\end{array}\to
\begin{array}{l}\sH\\\oplus\\\cH\end{array}
\]
is a selfadjoint contraction, $\cspan\left\{Y^n \wt M^*\sH:\;
n\in\dN_0\right\}=\cH$, and
\[
\wh B(z)=\wt C+z\wt M(I-zY)^{-1}\wt M^*.
\]
The minimal system $\cT$ is determined by $\wh B(z)$ uniquely up to
unitary equivalence (see \cite{AHS3}). For the derivative $\wh
B'(0)$ one has $\wh B'(0)=\wt M\wt M^*$. Now introduce
\[
 \sH_0:=\ker \wh B'(0)=\ker \wt M^*,\quad  B:=\wt C\uphar\sH_0.
\]
Then $B$ is a Hermitian contraction, $\wt C$ is an $sc$-extension of
$B$ in $\sH$, and $\wt B$ is an $sc$-extension of $B$ in
$\sH\oplus\cH.$ Notice, that $B$ is nondensely defined precisely
when
\[
 \sH_0\neq \sH \iff \wt M\neq 0 \iff \cH\neq \{0\}.
\]
Therefore, one can write (cf. \eqref{param1})
\[
 \wt C=\frac{1}{2}(B_M+B_\mu)
  +\frac{1}{2}(B_M-B_\mu)^{1/2}X_{11}(B_M-B_\mu)^{1/2}
\]
and
\[
\wt B=\begin{bmatrix}\dfrac{B_M+B_\mu}{2}&0\cr 0 &0
\end{bmatrix}
+\frac{1}{2}\begin{bmatrix}(B_M-B_\mu)^{1/2}&0\cr 0 &\sqrt{2}I
\end{bmatrix}\begin{bmatrix}X_{11}&X_{12}\cr X^*_{12} &Y
\end{bmatrix}\begin{bmatrix}(B_M-B_\mu)^{1/2}&0\cr 0 &\sqrt{2}I
\end{bmatrix},
\]
where $X_{11}$ is selfadjoint contraction in $\cran(B_M-B_\mu)$ and
$
\begin{bmatrix}X_{11}&X_{12}\cr X^*_{12} &Y
\end{bmatrix}
$ is a selfadjoint contraction in $\cran(B_M-B_\mu)\oplus\cH$. Thus
\[
\wt
B=\begin{bmatrix}\frac{1}{2}(B_M+B_\mu)+\frac{1}{2}(B_M-B_\mu)^{1/2}X_{11}(B_M-B_\mu)^{1/2}&
\frac{1}{\sqrt{2}}(B_M-B_\mu)^{1/2}X_{12}\cr
\frac{1}{\sqrt{2}}X^*_{12}(B_M-B_\mu)^{1/2}&Y\end{bmatrix}
\]
Hence $ \wt M=\frac{1}{\sqrt{2}}(B_M-B_\mu)^{1/2}X_{12},\; \wt
M^*=\frac{1}{\sqrt{2}}X^*_{12}(B_M-B_\mu)^{1/2}, $ and
\begin{multline*}
\wh B(z)=\wt
C+\frac{1}{2}(B_M-B_\mu)^{1/2}zX_{12}(I-zY)^{-1}X^*_{12}(B_M-B_\mu)^{1/2}\\
=\frac{1}{2}(B_M+B_\mu)+\frac{1}{2}(B_M-B_\mu)^{1/2}\left(X_{11}+zX_{12}(I-zY)^{-1}X^*_{12}\right)(B_M-B_\mu)^{1/2}.
\end{multline*}
Therefore, $\wh B(z)$ is of the form \eqref{BPHI}. Applying Theorem
\ref{comrescontr} and the formula \eqref{genres} one gets the first
statement of the theorem.

2) The function $\Theta$ can be realized as the transfer function of
the minimal passive selfadjoint system
\[
\Sigma=\left\{\begin{bmatrix}\wt C& \wt M\cr \wt M^*&  Y
\end{bmatrix},\;\cH,\cH, \sH\right\}
\]
with input-output space $\cH$ and the state space $\sH$. Again the
operator
\[
\wt B= \begin{bmatrix}\wt C& \wt M\cr \wt M^*&  Y
\end{bmatrix}:\begin{array}{l}\sH\\\oplus\\\cH\end{array}\to
\begin{array}{l}\sH\\\oplus\\\cH\end{array}
\]
is a selfadjoint contraction,
\begin{equation}
\label{minsys} \cspan\left\{\wt C^n\wt M\cH,\;
n\in\dN_0\right\}=\sH,
\end{equation}
and
\[
 \Theta(z)=Y+z\wt M^*(I_\sH-z \wt C)^{-1}\wt M,\quad
 z\in\dC\setminus\left\{(-\infty,-1]\cup[1,+\infty)\right\}.
\]
The minimal system $\Sigma$ is determined by $\Theta$ uniquely up to
unitary equivalence; see \cite{AHS3}. Define
\[
 \sN:=\cran\wt M,\quad \sH_0:=\sH\ominus\sN=\ker \wt M^*,
 \quad B:=\wt C\uphar\sH_0.
\]
Then $B$ is a Hermitian contraction, $\dom B=\sH_0$, and $\wt B$ is
an $sc$-extension of $B$. Moreover, \eqref{minsys} means that the
operator $B$ is simple, i.e., it has no reducing subspace on which
$B$ is selfadjoint. To complete the proof it remains to apply
Corollary \ref{compscreas}.
\end{proof}

The last part of this section is devoted to the study of the
following linear fractional transformation of the transfer function
$\Theta(z)$ of the form \eqref{transfunc}:
\begin{equation}\label{formula}
\begin{array}{l}
\cN(\lambda)=I_\cH-2
\left(I_\cH+\Theta\left(\cfrac{1+\lambda}{1-\lambda}\right)\right)^{-1}\\
=\left\{\left\{\left(I_\cH+\Theta\left(\cfrac{1+\lambda}{1-\lambda}\right)\right)h,\left(\Theta\left(\cfrac{1+\lambda}{1-\lambda}\right)-I_\cH\right)h\right\},\;
h\in\cH\right\},
\end{array}
\end{equation}
where $\lambda\in \dC\setminus[0,+\infty)$. From the properties in
\eqref{propert1} it follows that for all $\beta\in (0,\pi/2)$
\[
 \arg\lambda\in[\pi-\beta,\pi+\beta] \;\Longrightarrow \left|\IM
 (\cN(\lambda)f,f)_\cH\right|\le\tan\beta\,\RE(\cN(\lambda)f,f)_\cH,
 \; f\in\dom \cN(\lambda).
\]
Hence, the linear relation $\cN(\lambda)$ is $m$-sectorial for each
$\RE \lambda<0$ and, in particular, if $\lambda<0$ then
$\cN(\lambda)$ is nonnegative and selfadjoint. In the next theorem
the main analytic properties of $\cN(\lambda)$ are established and
an explicit representation for $\cN(\lambda)$ is obtained. Using the
terminology in \cite{Ka} the result shows in particular that
$\cN(\lambda)$ forms a holomorphic family of the type (B) in the
left open half-plane.

\begin{theorem}\label{repweyl11}
The domain $\sL:=\cD[\cN(\lambda)]$ of the closed form
$\cN(\lambda)[\cdot,\cdot]$ associated with the family
$\cN(\lambda)$ in \eqref{formula} does not depend on $\lambda,$
$\RE\lambda<0$, and the form $\cN(\lambda)[h,g]$ admits the
representation
\begin{multline*}
\cN(\lambda)[h,g]\\
=\left(\left(I_{\cH}+V(\wh B_1-\wh B_0)^{1/2}\left(\wh
 B_0-\frac{1-\lambda}{1+\lambda}\, I_\sH\right)^{-1}(\wh B_1-\wh B_0)^{1/2}V^*\right)Yh, Yg\right)
,\\
\RE\lambda<0,\;h,g\in \sL:=\ran(I_\cH+\Theta(0))^{1/2},
\end{multline*}
where $\wh B_0$ and $\wh B_1$ are as defined in \eqref{BPHIPM} and
\eqref{BPHIPM1},
\[
 Y=(I_\cH-\Theta(0))^{1/2}(I_\cH+\Theta(0))^{(-1/2)}:\sL\to \overline{\sL},
\]
and $V:\cran(\wh B_1-\wh B_0)\to\cH$ is an isometry. Here
$(I_\cH+\Theta(0))^{(-1/2)}$ is the Moore-Penrose pseudo inverse.
\end{theorem}
\begin{proof} Since $\Theta$ is the transfer function of a
passive selfadjoint discrete-time system, see \eqref{transfunc},
$||\Theta(z)||\le 1$ for all $|z|<1$ and $\Theta^*(z)=\Theta(\bar
z)$. Then the real part
$$\RE(\Theta(z))=\frac{1}{2}(\Theta(z)+\Theta^*(z))$$
satisfies $I_\cH\pm\RE(\Theta(z))\ge 0$ for all $z\in\dD$. Since
$I_\cH\pm\RE(\Theta(z))$ are harmonic functions, a result of Yu.L.
Shmul'yan \cite{Shmul2} yields the following invariance equalities
\[
\begin{array}{l}
\ran(I_\cH+\RE\Theta(z))^{1/2}=\ran(I_\cH+\Theta(0))^{1/2},\\
\ran(I_\cH-\RE\Theta(z)))^{1/2}=\ran(I_\cH-\Theta(0))^{1/2},\\
\end{array}
\]
for all $z\in\dD$; observe that $\Theta(0)=\Theta(0)^*$. From
Douglas Theorem \cite{Doug} we get
\[
 (I_\cH+\RE\Theta(z))^{1/2}=(I_\cH+\Theta(0))^{1/2}F(z),
\]
where $F^{-1}(z)$ is bounded for all $z\in\dD$ in
$\cran(I_\cH+\Theta(0))$. Since $\Theta(z)\in\wt C_\cH$ for all
$z\in\dD$, the operators $I_\cH+\Theta(z)$ are $m$-sectorial bounded
operators. Therefore,
\[
I_\cH+\Theta(z)=(I_\cH+\Theta_R(z))^{1/2}(I+iG(z))(I_\cH+\Theta_R(z))^{1/2},\;
z\in\dD,
\]
where $G(z)=G^*(z)$ in the subspace $\cran(I_\cH+\Theta_R(0))$ and
$I$ is the identity operator in $\cran(I_\cH+\Theta_R(0))$. Hence
\[
I_\cH+\Theta(z)=(I_\cH+\Theta_R(0))^{1/2}F(z)(I+iG(z))F^*(z)(I_\cH+\Theta_R(0))^{1/2},\;
z\in\dD.
\]
In addition, the function $\Theta$ can be represented in the form
(see \cite{Shmul2})
\[
\Theta(z)=\Theta(0)+D_{\Theta(0)}\Phi(z)D_{\Theta(0)},\; z\in\dD,
\]
where $\Phi(z)$ is holomorphic in $\dD.$ Since $
D_{\Theta(0)}=(I_\cH+\Theta(0))^{1/2}(I_\cH-\Theta(0))^{1/2} $ we
obtain
\[
\Theta(z)=\Theta(0)+(I_\cH+\Theta(0))^{1/2}\Psi(z)(I_\cH+\Theta(0))^{1/2},\;
z\in\dD,
\]
where
\[
\Psi(z)=(I_\cH-\Theta(0))^{1/2}\Phi(z)(I_\cH-\Theta(0))^{1/2}.
\]
On the other hand,
\begin{equation}\label{Psi}
 I_\cH+\Theta(z)=(I_\cH+\Theta(0))^{1/2}(I+\Psi(z))(I_\cH+\Theta(0))^{1/2}.
\end{equation}
Thus $I+\Psi(z)=F(z)(I+iG(z))F^*(z).$ It follows that $I+\Psi(z)$
has bounded inverse in $\cran(I_\cH+\Theta(0))^{1/2}$. Furthermore
we use Proposition \ref{clfrm}. For $\lambda$ with $\RE\lambda<0$ we
get
\[
\cD[\cN(\lambda)]=\ran
\left(I_\cH+\RE\Theta\left(\cfrac{1+\lambda}{1-\lambda}\right)\right)^{1/2}
=\ran(I_\cH+\Theta(0))^{1/2},\;\RE\lambda<0.
\]
Consequently, the domain $\cD[\cN(\lambda)]$ of the closed sectorial
form $\cN(\lambda)[\cdot,\cdot]$ is constant if $\RE\lambda<0$. For
$u\in \ran(I_\cH+\Theta(0))^{1/2}$ if $\RE \lambda<0$ and
$\lambda=(z-1)(z+1)^{-1}$ we get
\begin{equation}\label{cN}
\begin{array}{l}
\cN(\lambda)[u]=-||u||^2+2((I+iG(z))^{-1}(I_\cH+\RE\Theta(z))^{-1/2}u,(I_\cH+\RE\Theta(z))^{-1/2}u)\\
=-||u||^2+2((I+iG(z))^{-1}F^{-1}(z)(I_\cH+\Theta(0))^{-1/2}u,F^{-1}(z)(I_\cH+\Theta(0))^{-1/2}u)\\
=-||u||^2+2((I+\Psi(z))^{-1}(I_\cH+\Theta(0))^{-1/2}u,
(I_\cH+\Theta(0))^{-1/2}u).
\end{array}
\end{equation}
Therefore, $\cN(\lambda)[u]$ is holomorphic in $\lambda$ in the left
half-plane. Consequently,  $\cN(\lambda)$ forms a holomorphic family
of type (B) in the left half-plane in the sense of \cite{Ka}.

Next the representation of the form $\cN(\lambda)[\cdot,\cdot]$ is
derived. Let $\wt B=\wt B_X$ be as in \eqref{CONTREXT1}, let $\wh
Z_0$ and $\wh Z_1$ be given by \eqref{predeli}, and let $\wh B_0$
and $\wh B_1$ be given by \eqref{BPHIPM} and \eqref{BPHIPM1},
respectively. Then using the representation
\[
X=\begin{bmatrix} X_{11}& UD_{X_{22}}\cr
D_{X_{22}}U^*&X_{22}\end{bmatrix}:\begin{array}{l}\sN\\\oplus\\\cH\end{array}\to\begin{array}{l}\sN\\\oplus\\\cH\end{array},
\]
where $U\in\bL(\sD_{X_{22}},\sN)$ is a contraction, see Remark
\ref{parr}, one can write
\[
\wh Z_0=X_{11}-U(I_{\sD_{X_{22}}}-X_{22})U^*,\quad \wh
Z_1=X_{11}+U(I_{\sD_{X_{22}}}+X_{22})U^*.
\]
Moreover, $\wh B_1-\wh B_0=2D_{K^*_0}UU^*D_{K^*_0}P_{\sN}$ and if
$\wt C$ is an in \eqref{opc} then
\[
 \wt C-\wh B_0=D_{K^*_0}(X_{11}-\wh Z_0)D_{K^*_0}P_{\sN}
 =D_{K^*_0}U(I_{\sD_{X_{22}}}-X_{22})U^*D_{K^*_0}P_{\sN}.
\]

Define
\[
Q_{\wt C}(\xi)=P_{\sN}(\wt C -\xi I_\sH)^{-1}\uphar\sN,\quad
\xi\in\rho(\wt C).
\]
Then it follows from \eqref{Sh-Fr1} that
\[
Q_{\wt C}(\xi)=Q_{\wh B_0}(\xi)\left(I_{\sN}+(\wt C-\wh B_0)Q_{\wh
B_0}(\xi)\right)^{-1},\quad \xi\in\dC\setminus [-1,1],
\]
cf. \cite{AHS2}. Furthermore, for $\xi\in\dC\setminus [-1,1]$
\[
\begin{array}{l}
X^*_{12}D_{K^*_0}Q_{\wt C}(\xi)D_{K^*_0}X_{12}=D_{X_{22}}U^*D_{K^*_0}Q_{\wt C}(\xi)D_{K^*_0}UD_{X_{22}}\\
\quad= D_{X_{22}}U^*D_{K^*_0}Q_{\wh
B_0}(\xi)\left(I_{\sN}+D_{K^*_0}U(I_{\sD_{X_{22}}}-X_{22})U^*D_{K^*_0}Q_{\wh
B_0}(\xi)\right)^{-1}D_{K^*_0}UD_{X_{22}}\\
\quad=(I_\cH+X_{22})^{1/2}\left(I_{\cH}+(I_\cH-X_{22})^{1/2}U^*D_{K^*_0}Q_{\wh
B_0}(\xi)D_{K^*_0}UP_{\sD_{X_{22}}}(I_\cH-X_{22})^{1/2}\right)^{-1}\\
\qquad\times(I_\cH-X_{22})^{1/2}U^*D_{K^*_0}Q_{\wh
B_0}(\xi)D_{K^*_0}UD_{X_{22}},
\end{array}
\]
where $P_{\sD_{X_{22}}}$ is the orthogonal projection in $\cH$ onto
$\sD_{X_{22}}$ and the last identity follows from
\[
\begin{array}{l}
 \left(I_{\cH}+(I_\cH-X_{22})^{1/2}U^*D_{K^*_0}Q_{\wh B_0}(\xi)D_{K^*_0}UP_{\sD_{X_{22}}}(I_\cH-X_{22})^{1/2}\right)
  (I_\cH-X_{22})^{1/2}U^*D_{K^*_0}Q_{\wh B_0}(\xi) \\
 \quad=(I_\cH-X_{22})^{1/2}U^*D_{K^*_0}Q_{\wh B_0}(\xi)
  \left(I_{\sN}+D_{K^*_0}U(I_{\sD_{X_{22}}}-X_{22})U^*D_{K^*_0}Q_{\wh B_0}(\xi)\right).
\end{array}
\]
This yields, see \eqref{transfunc},
\[
\begin{array}{l}
 I_ \cH+\Theta\left(1/\xi\right)=I_\cH+X_{22}-X^*_{12}D_{K^*_0}Q_{\wt C}(\xi)D_{K^*_0}X_{12}=\\
 (I_\cH+X_{22})^{1/2}\left(I_\cH+(I_\cH-X_{22})^{1/2}U^*D_{K^*_0}Q_{\wh
B_0}(\xi)D_{K^*_0}UP_{\sD_{X_{22}}}(I_{\cH}-X_{22})^{1/2}\right)^{-1}
 (I_\cH+X_{22})^{1/2}.
\end{array}
\]
Since $\Theta(0)=X_{22}$, it follows from \eqref{Psi} that
\[
 I+\Psi\left(1/\xi\right)=\left(I+(I-X_{22})^{1/2}U^*D_{K^*_0}Q_{\wh
B_0}(\xi)D_{K^*_0}UP_{\sD_{X_{22}}}(I-X_{22})^{1/2}\right)^{-1}\uphar\cran(I_\cH+\Theta(0)),
\]
where $I=I_{\cran(I_\cH+\Theta(0))}$. The equality $\wh B_1-\wh
B_0=2D_{K^*_0}UU^*D_{K^*_0}P_{\sN}$ implies that
\[
\sqrt{2} U^*D_{K^*_0}P_{\sN}=V(\wh B_1-\wh B_0)^{1/2},
\]
holds for some isometry $V$ mapping $\cran(\wh B_1-\wh B_0)$ onto
$\cran U^*(\subseteq\sD_{X_{22}}\subseteq\cH)$. Hence,
\[
\begin{array}{l}
\left(I+\Psi\left(1/\xi\right)\right)^{-1}
 =I+(I-X_{22})^{1/2}U^*D_{K^*_0}Q_{\wh B_0}
 (\xi)D_{K^*_0}UP_{\sD_{X_{22}}}(I-X_{22})^{1/2}\\
 =I+\frac{1}{2}(I-X_{22})^{1/2}V(\wh B_1-\wh B_0)^{1/2}Q_{\wh
B_0}(\xi)(\wh B_1-\wh B_0)^{1/2}V^*(I-X_{22})^{1/2}\\
 =\frac{1}{2}(I+X_{22})\\
\quad +\frac{1}{2}(I-X_{22})^{1/2}\left(I+V(\wh
 B_1-\wh B_0)^{1/2}Q_{\wh B_0}(\xi)(\wh B_1-\wh
 B_0)^{1/2}V^*\right)(I-X_{22})^{1/2}.
\end{array}
\]
It remains to substitute this expression into the representation of
$\cN(\lambda)$ in \eqref{cN} to conclude that for
$h,g\in\ran(I+X_{22})^{1/2}$ and for $\RE \lambda<0$ with
$\xi=(1-\lambda)(1+\lambda)^{-1}$,
\[
\cN(\lambda)[h,g] =\left(\left(I+V(\wh B_1-\wh B_0)^{1/2}Q_{\wh
B_0}(\xi)(\wh B_1-\wh B_0)^{1/2}V^*\right)Yh,Yg\right),
\]
where
\[
Y =(I-X_{22})^{1/2}(I+X_{22})^{(-1/2)}
=(I_\cH-\Theta(0))^{1/2}(I_\cH+\Theta(0))^{(-1/2)}:\sL\to
\overline{\sL}.
\]
This completes the proof.
\end{proof}


\bigskip
\noindent \textbf{Acknowledgement} The first author thanks the
V\"ais\"al\"a Foundation for the financial support during his stay
at the University of Vaasa in 2014. The second author is grateful
for the support from the Emil Aaltonen Foundation.

\end{document}